\documentclass{article}
[12pt , a4paper]
\usepackage[T1]{fontenc}
\usepackage[english]{babel}
\usepackage{amsmath}
\usepackage{amssymb,amsthm}
\usepackage{mathrsfs}
\usepackage{marvosym}
\usepackage{amsfonts}
\usepackage{yfonts}
\usepackage{graphicx}
\usepackage{eurosym}
\usepackage{enumerate}
\usepackage{makeidx}
\usepackage{multicol}
\usepackage{stmaryrd}
\usepackage[all]{xy}
\usepackage{fancyhdr}
\usepackage[pdftex]{hyperref}
\usepackage{color}
\DeclareGraphicsExtensions{.jpg,.pdf,.mps,.png}
\definecolor{R}{rgb}{0.5, 0, 0}
\definecolor{B}{rgb}{0, 0, 0.5}
\usepackage[margin=3cm]{geometry}
\usepackage[bottom]{footmisc}

\usepackage[small,hang]{caption}

\newtheoremstyle{note}
{4 mm}
{1 mm}
{\itshape}
{}
{\bf\sffamily}
{.}
{.5em}
{}

\theoremstyle{note}
\newtheorem{thm}{Theorem}
\newtheorem{prop}[thm]{Proposition}
\newtheorem{cor}[thm]{Corollary}
\newtheorem{lem}[thm]{Lemma}

\newtheoremstyle{note2}
{4 mm}
{1 mm}
{}
{}
{\bf\sffamily}
{.}
{.5em}
{}

\theoremstyle{note2}
\newtheorem{rem}[thm]{Remark}

\newtheorem{df}[thm]{Definition}

\DeclareMathOperator{\Id}{Id}

\DeclareMathOperator{\Ad}{Ad}
\DeclareMathOperator{\ad}{ad}

\DeclareMathOperator{\Aut}{Aut}

\DeclareMathOperator{\Der}{Der}
\DeclareMathOperator{\Op}{Op}

\DeclareMathOperator{\fa}{\mathfrak{a}}

\DeclareMathOperator{\fg}{\mathfrak{g}}

\DeclareMathOperator{\fk}{\mathfrak{k}}

\DeclareMathOperator{\fm}{\mathfrak{m}}
\DeclareMathOperator{\fn}{\mathfrak{n}}

\DeclareMathOperator{\fp}{\mathfrak{p}}

\DeclareMathOperator{\fs}{\mathfrak{s}}

\DeclareMathOperator{\CE}{CE}
\DeclareMathOperator{\op}{\Op^{\mathbb{S}}\left(\ast_\nu\right)}
\DeclareMathOperator{\rank}{rank}

\title{\textbf{Quantum moment maps and symmetric \\ bounded domains quantizations}} \date{January 2017 \\ \hspace{1 mm} \vspace{-5 mm} \\ \textit{reviewed in} June 2018} \author{{\Large \textsc{St\'ephane Korvers}} \\ \hspace{1 mm} \vspace{-3 mm} \\ \small{Universit\'e du Luxembourg (2015-2017)} \\ \hspace{1 mm} \vspace{-7 mm} \\ \small{FSTC, Unit\'e de Recherche en Math\'ematiques} \\ \hspace{1 mm} \vspace{-7 mm} \\ \small{rue Richard Coudenhove-Kalergi, 6} \\ \hspace{1 mm} \vspace{-7 mm} \\ \small{L-1359 Luxembourg, Grand Duchy of Luxembourg} \\ \hspace{1 mm} \vspace{-4 mm} \\ {\small \textit{E-mail:} korvers.s@gmail.com}}

\begin{document}


\renewcommand{\proofname}{\textbf{Proof.}}
\renewcommand{\qedsymbol}{\Squaresteel}
\renewcommand{\labelitemi}{$\bullet$}

\pagestyle{fancy} 
\fancyhf{} 
\cfoot{\thepage}
\lhead{Quantum moment maps and symmetric bounded domains quantizations}
\rhead{S. Korvers}

\thispagestyle{empty}

\maketitle

\vspace{1 mm}

\begin{center} \small
\textbf{Abstract}

\textit{We introduce an explicit construction for realizing of the space of invariant \\ deformation quantizations on an arbitrary symmetric bounded domain of $\mathbb{C}^n$.}
\end{center} \normalsize

\vspace{1 mm}

\section{Introduction and notations}

Generally speaking, in the context of mathematical physics and quantum mechanics, the terminology of \emph{quantization} is used to allude to the expression at a quantum level of facts related to a classical system. This problem starts with the data of a symplectic manifold $\left(M, \omega\right)$, or more generally a Poisson manifold $\left(M, \left\{-, -\right\}\right)$, modeling the phase space of the classical system. Usually, by quantizing $\left(M, \omega\right)$, one asks for a way to link some classical objects to potential quantum analogs. For example, with the symplectic manifold $\left(M, \omega\right)$ and the algebra of smooth functions on $M$ representing the classical observables, we can respectively associate a Hilbert space $\mathcal{H}$ and an algebra of linear operators on $\mathcal{H}$. Many methods exist to approach this problem. Among them, the \emph{deformation quantization} promotes the idea of an understanding of this quantization problem as a deformation of the commutative structure of the algebra of classical observables $\mathcal{C}^\infty\left(M\right) := \mathcal{C}^\infty\left(M, \mathbb{C}\right)$ into a noncommutative direction given by the Poisson bracket $\left\{-, -\right\}$ associated with the symplectic form $\omega$. 

\noindent At a formal level, this notion is encoded in the data of a \emph{star-product} on $M$ which is an associative $\mathbb{C}\llbracket\nu\rrbracket$-linear product on the space of formal power series in the formal parameter $\nu$ with coefficients in $\mathcal{C}^\infty\left(M\right)$
$$\ast_\nu : \mathcal{C}^\infty\left(M\right)\llbracket\nu\rrbracket \times \mathcal{C}^\infty\left(M\right)\llbracket\nu\rrbracket \rightarrow \mathcal{C}^\infty\left(M\right)\llbracket\nu\rrbracket : \left(f_1, f_2\right) \mapsto f_1 \ast_\nu f_2 := \sum_{k \in \mathbb{N}} \nu^k C_k\left(f_1, f_2\right)$$ where $\left\{C_k : \mathcal{C}^\infty\left(M\right) \times \mathcal{C}^\infty\left(M\right) \rightarrow \mathcal{C}^\infty\left(M\right)\right\}_{k \in \mathbb{N}}$\, are a bi-differential operators such that 
$$C_0\left(f_1, f_2\right) = f_1 f_2, \,\,\,\,C_1\left(f_1, f_2\right) - C_1\left(f_2, f_1\right) = 2 \left\{f_1, f_2\right\} \text{ \, and \, } C_l\left(1, f_1\right) = C_l\left(f_1, 1\right) = 0$$
for each $l \in \mathbb{N}\backslash\left\{0\right\}$ and $f_1, f_2 \in \mathcal{C}^\infty\left(M\right)$. This was introduced by Bayen, Flato, Fronsdal, Lichnerowicz and Sternheimer in 1978; \cite{B+78a}, \cite{B+78b}. This approach has the property to be universal in the sense that there exists a star-product on each Poisson manifold; \cite{Ko03}. 

\noindent Although one does not worry about the convergence of the formal series appearing in the previous definition, under adapted functional hypothesis, it may happen that a new function on $M$ can be defined from the star-product of two functions on $M$. In this case, one talks about \emph{non-formal deformation quantization on $M$}. More specifically, we are interested in such quantization described by an explicit 3-point kernel $K_\nu\left(-, -, -\right)$ through the formula 
$$\left(f \ast_\nu g\right)\left(x\right) \,=\, \int_{M \times M} \,K_\nu\left(x, y, z\right) \,f\left(y\right) \,g\left(z\right) \,dy\,dz$$
when $f$ and $g$ belongs to an adapted space of functions, with $dx$ the Liouville measure on $M$.

\noindent In some situations, it is relevant to look for specific deformation quantizations that take account of symmetries of a classical system through the quantization process. If $G$ is a Lie group which acts by symplectomorphisms on the symplectic manifold $\left(M, \omega\right)$ through the action map \,$$\tau : G \times M \rightarrow M : \left(g, x\right) \mapsto \tau_g\left(x\right),$$ a star-product $\ast_\nu$ on $M$ will be said to be \emph{$G$-invariant} if 
\begin{eqnarray}\label{4}
\tau_g^\star \left(f_1 \ast_\nu f_2\right) = \tau_g^\star f_1 \ast_\nu \tau_g^\star f_2
\end{eqnarray}
for each $g \in G$ and $f_1, f_2 \in \mathcal{C}^\infty\left(M\right)$. When $G$ preserves a symplectic connexion on $M$, then there always exists a $G$-invariant star-product on $M$. It is a consequence of the well known Fedosov construction of star-products on symplectic manifolds; \cite{Fe94}. 

All along this text, we will consider $\mathbb{D} \subset \mathbb{C}^N$ an arbitrary \emph{symmetric bounded domain} of $\mathbb{C}^N$ for $N \in \mathbb{N}\backslash\left\{0\right\}$, ie. an open connected bounded subset of $\mathbb{C}^N$ endowed with a structure of symmetric space for which the symmetries are biholomorphisms. Such domain is connected simply connected; \cite[Ch.\,8, thm.\,4.6]{He01}. When it is endowed with its Bergman metric, it has a structure of an Hermitian symmetric space of non compact type; \cite[Ch.\,8, thm.\,7.1]{He01}. In addition, every Hermitian symmetric space of non compact type can be realized as a symmetric bounded domain; \cite[Ch.\,8, thm.\,7.1]{He01}. As before, we will denote by $\omega$ and $\left\{-, -\right\}$ respectively the symplectic structure on $\mathbb{D}$ and the Poisson bracket on $\mathcal{C}^\infty\left(\mathbb{D}\right)$ associated with $\omega$. 

\noindent Let $G$ be the identity component of the automorphism group of $\mathbb{D}$ and $\fg$ its Lie algebra. It is well known that $G$ is a semi-simple Lie group of transformations of $\mathbb{D}$ which acts holomorphically and transitively on $\mathbb{D}$; \cite[Ch.\,4 \& 8]{He01}, \cite[Ch.\,1, \S\,2]{Ko14}. We will denote by $$\tau : G \times \mathbb{D} \rightarrow \mathbb{D} : \left(g, x\right) \mapsto \tau_g\left(x\right)$$ the action of $G$ on $\mathbb{D}$. For $X \in \fg$, the notation $X^\star \in \Gamma\left(T\mathbb{D}\right)$ will refers to the \emph{fundamental vector field} associated with $X$ which is defined at point $x \in \mathbb{D}$ by $$X^\star_x := \left.\frac{d}{dt}\right|_{t=0} \tau_{\exp\left(- t X\right)}\left(x\right).$$ The action of $G$ on $\mathbb{D}$ is Hamiltonian and admits a unique \emph{(co-)moment map} $$\lambda : \fg \rightarrow \mathcal{C}^\infty\left(\mathbb{D}\right) : X \mapsto \lambda_X$$ defined by the equality $\iota_{X^\star}\omega = - d\lambda_X$ for each $X \in \fg$; \cite[Ch.\,2, thm.\,11.8]{So97}, \cite[Ch.\,26]{Ca08}. In particular, for each $X, Y \in \fg$, we have $$X^\star = \left\{\lambda_X, -\right\} : f \in \mathcal{C}^\infty\left(\mathbb{D}\right) \mapsto \left\{\lambda_X, f\right\} \text{\,\,\, and \,\,\,} X^\star\left(\lambda_Y\right) = \lambda_{\left[X, Y\right]}.$$ 

In the present work, we develop a method unifying constructions of $G$-invariant star-products on $\mathbb{D}$. We present a characterization of the space of all these invariant star-products as solutions to an explicit hierarchy of partial differential equations and we explicit how to write these equations. 

\noindent The method used in this work combines modern mathematics of various research fields in an innovative way, and is based on the \emph{retract method} initiated by Bieliavsky and his collaborators in the 2000s. It have already proven its power in the obtention of similar descriptions for the particular cases of the Poincar\'e disk and the unit ball of $\mathbb{C}^N$; \cite{B+09}, \cite{Ko14}. In the following sections, we show that a similar approach can be performed under hypothesis that we describe. We also develop tools for simplifying computations underlined by practical applications of this method.

\vspace{3 mm}

\noindent \underline{\texttt{Acknowledgement}} \small

\vspace{-2 mm}

\noindent \texttt{This work is supported by the \emph{Fonds National de la Recherche FNR/AFR-Postdoc grant} no.8960322.}

\vspace{-3 mm}

\noindent \texttt{The author thanks the Fonds National de la Recherche, the University of Luxembourg, and Martin \\ Schlichenmaier for supporting him and giving him the opportunity to pursue his research in a \\stimulating international working environment.}

\vspace{-3 mm}

\noindent \texttt{The author thanks Pierre Bieliavsky for introducing him to this field, and for inspiring this \\ quantization method through our collaboration in \cite{Ko14}.} \normalsize

\vspace{10 mm}

\section{Structure of the automorphism group of $\mathbb{D}$}

In this section, we describe the structure of the automorphism group of $\mathbb{D}$ and its Lie algebra $\fg$. In particular, we explicit the restricted root space decomposition and the Pyatetskii-Shapiro decomposition of $\fg$. We show that the domain $\mathbb{D}$ can be identify with the Iwasawa group of $G$.

\subsection{Root space decomposition}

Let's fix $o \in \mathbb{D}$. Then, the subgroup $K := \left\{g \in G \,|\, \tau_g\left(o\right) = o\right\} \subset G$ is compact and the map $$G/K \rightarrow \mathbb{D} : g K \mapsto \tau_g\left(o\right)$$ is a diffeomorphism; \cite[Ch.\,4, thm.\,3.3]{He01}. As the domain $\mathbb{D}$ has a structure of Hermitian symmetric space of non compact type, the Lie algebra $\fg$ admits a Cartan decomposition $$\fg = \fk \oplus \fp$$ where $\fk$ is the Lie algebra of $K$ and $\fp$ is invariant under the adjoint action of $K$; \cite[Ch.\,8, \S\,4]{He01}. Let's denote by $$\sigma = \Id_{\fk} \oplus -\Id_{\fp} : \fg \rightarrow \fg$$ the associated Cartan involution and $\beta$ the Killing form of $\fg$. Then, the symmetric bilinear form $$\beta_\sigma : \left(X, Y\right) \in \fg \times \fg \mapsto - \beta\left(X, \sigma\left(Y\right)\right)$$ is positive definite and $\beta\left(X, Y\right) = 0$ for each $X \in \fk$ and $Y \in \fp$. Let's consider $\fa$ an abelian Lie subalgebra of $\fg$ contained in $\fp$ and maximal for this property. We set $$r := \dim\left(\fa\right)$$ to be the \emph{rank} of $\mathbb{D}$. This number is independent from the choice of $\fa$; \cite[Ch.\,6, thm.\,6.51]{Kn02}. For each linear form $\left[\lambda : \mathfrak{a} \rightarrow \mathbb{R}\right] \in \mathfrak{a}^\star$, we can define $$\mathfrak{g}_\lambda := \left\{X \in \mathfrak{g} \,:\, \left[H, X\right] = \lambda\left(H\right) X \text{ for each } H \in \mathfrak{a}\right\} \subset \mathfrak{g}.$$ 

\begin{df}
A linear form $\lambda \in \mathfrak{a}^\star \backslash \left\{0\right\}$ such that $\mathfrak{g}_\lambda$ is non trivial will be called \emph{(restricted) root of $\fg$}. The set of all these roots will be denoted by $\Sigma \subset \mathfrak{a}^\star$. For $\lambda \in \Sigma$, the subspace $\fg_\lambda$ is called \emph{(restricted) root space} of $\fg$.
\end{df}

\begin{prop}\label{rsd} \emph{\cite[Ch.\,6, prop.\,6.40]{Kn02}} 
The Lie algebra $\fg$ admits a root space decomposition $$\mathfrak{g} \,=\, \mathfrak{g}_0 \,\oplus\, \left(\bigoplus_{\lambda \in \Sigma} \mathfrak{g}_\lambda\right).$$
For each $\lambda, \mu \in \mathfrak{a}^\star$, we have $\left[\mathfrak{g}_\lambda, \mathfrak{g}_\mu\right] \subset \mathfrak{g}_{\lambda + \mu}$ and $\mathfrak{g}_{-\lambda} = \sigma\left(\mathfrak{g}_\lambda\right)$. In addition, the subspace $\fg_0$ is a Lie subalgebra of $\fg$ which admits a decomposition $$\fg_0 = \fa \oplus \fm \text{ \,\,with\,\,\, } \mathfrak{m} := \left\{X \in \fk : \left[H, X\right] = 0 \text{ for each } H \in \mathfrak{a}\right\}.$$
\end{prop}

\noindent The root space decomposition of $\fg$ is an orthogonal direct sum for the inner product $\beta_\sigma$ given that $\beta\left(X, Y\right) = 0$ for each $X \in \fg_\lambda$ and $Y \in \fg_\mu$ if $\lambda, \mu \in \mathfrak{a}^\star$ satisfy $\lambda + \mu \neq 0$.

The Lie algebra $\fm$ admits the decomposition 
\begin{eqnarray}\label{1c}
\fm = \left[\fm, \fm\right] \oplus Z\left(\fm\right)
\end{eqnarray}
where $Z\left(\fm\right)$ denotes the center of $\fm$; \cite[Ch.\,1, cor.\,1.56 \& Ch.\,7, prop.\,7.48]{Kn02}. As $\beta$ is positive definite on $\fa \times \fa$, for $\lambda \in \fa^\star$, we can define $H_\lambda \in \fa$ as the unique element in $\fa$ such that $\beta\left(H_\lambda, H\right) = \lambda\left(H\right)$ for each $H \in \fa$. The set $\left\{H_\lambda \,:\, \lambda \in \Sigma\right\}$ spans $\fa$; \cite[Ch.\,6, cor.\,6.53]{Kn02}. For all $\lambda \in \Sigma$ and $X \in \fg_\lambda$, we have 
\begin{eqnarray}\label{1}
\left[X, \sigma\left(X\right)\right] = \beta\left(X, \sigma\left(X\right)\right) H_\lambda \,;
\end{eqnarray}
\cite[Ch.\,6, prop.\,6.52]{Kn02}. Let's notice that $\beta\left(X, \sigma\left(X\right)\right) < 0$ if $X \neq 0$ in the previous equality, as $\beta_\sigma$ is positive definite. 

\noindent We conduce this section with the following technical lemma.

\begin{lem}\label{m} \emph{\cite[thm.\,1]{Ko16}} 
Let's consider $\lambda \in \Sigma$ and $X \in \fg_\lambda\backslash\left\{0\right\}$. Then, we have $$\left[\fm, X\right] = X^{\perp\left(\lambda\right)} := \left\{Y \in \fg_\lambda : \beta_\sigma\left(X, Y\right) = 0\right\}.$$ In particular, the root space $\fg_\lambda$ admits the decomposition $\fg_\lambda = \mathbb{R} X \oplus \left[\fm, X\right]$.
\end{lem} 

\noindent As a consequence of this lemma, if $\lambda \in \Sigma$ is such that $\dim\left(\fg_\lambda\right) = 1$, then $\left[\fm, X\right] = 0$. In addition, if $\fm = 0$, all the root spaces of $\fg$ are one-dimensional.

\subsection{Iwasawa decomposition}

Let's fix $\left\{\varphi_1, ..., \varphi_r\right\}$ a basis of $\mathfrak{a}^\star$. We will say that the root $\lambda \in \Sigma$ is \emph{positive} if there exists $1 \leq k_0 \leq r$ such that $\varphi_{k_0}\left(H_\lambda\right) > 0$ and $\varphi_{k}\left(H_\lambda\right) = 0$ for each $k < k_0$. We will denote by $\Sigma^+$ the set of positive roots of $\fg$. Now, we can introduce the \emph{Iwasawa decomposition} of $\fg$ and $G$.

\begin{prop}\label{iwa2} \emph{\cite[Ch.\,6, prop.\,6.43 \& thm.\,6.46]{Kn02}} The Lie algebra $\mathfrak{g}$ admits the following vector space decomposition
$$\mathfrak{g} = \mathfrak{a} \oplus \mathfrak{n} \oplus \mathfrak{k} \text{ \,\,\,with\, } \fn := \bigoplus_{\lambda \in \Sigma^+} \mathfrak{g}_\lambda.$$
The connected Lie subgroup $A \subset G$ (resp.\,$N \subset G$) which Lie algebra $\mathfrak{a}$ (resp.\,$\mathfrak{n}$) is abelian (resp.\,nilpotent) an simply connected. The group $$\mathbb{S} := AN$$ is a connected simply connected solvable Lie subgroup of $G$ called Iwasawa group of $G$. In addition, the maps
$$A \times N \rightarrow \mathbb{S} : \left(a, n\right) \mapsto an \,\,\,\,\, \text{ and } \,\,\,\,\, \mathbb{S} \times K \rightarrow G : \left(s, k\right) \mapsto sk$$
are global diffeomorphisms between smooth manifolds.
\end{prop}

\noindent As a consequence, we get a diffeomorphism $\mathbb{S} \rightarrow G/K : s \mapsto s K$. In particular, the action of the Iwasawa group $\mathbb{S}$ on the symmetric bounded domain $\mathbb{D}$ is simply transitive and we have the identification $\mathbb{S} \simeq \mathbb{D}$. Let's extend the notation $\tau$ to denote the $G$-equivariant transport of this action on $\mathbb{S} \simeq G/K$. It is easy to notice that $$\tau_s\left(s^\prime\right) = s s^\prime =: L_s\left(s^\prime\right)$$ for each $s, s^\prime \in \mathbb{S}$. In particular, through its identification with $\mathbb{D}$, the group $\mathbb{S}$ becomes a left-invariant K\"ahlerian Lie group. 

In this text, we will denote by $\fs$ the Lie algebra of $\mathbb{S}$. We have the following vector space isomorphisms:\, $\fs \,\simeq\, \fa \oplus \fn \,\simeq\, \fp \,\simeq\, T_o\left(\mathbb{D}\right).$
We can notice the identities $$\left[\mathfrak{s}, \mathfrak{s}\right] = \mathfrak{n} \text{ \,\,\,and\,\,\, } N\left(\fn\right) = \fs \oplus \fm$$ where $N\left(\fn\right)$ is the normalizer of $\fn$ in $\fg$. The first equality and the inclusion $\left[\fg_0 \oplus \fn, \fn\right] \subset \fn$ are direct from the properties of root space decomposition of $\mathfrak{g}$. As a consequence, the second equality follows from (\ref{1}).

\noindent The Iwasawa decompositions of $\mathfrak{g}$ and $G$ can be written $$\mathfrak{g} = \mathfrak{s} \oplus \mathfrak{k} \text{ \,\,\,and\,\,\, } G = \mathbb{S} K \simeq \mathbb{S} \times K$$ respectively. We will denote the associated decompositions of $X \in \fg$ and $g \in G$ respectively by $$X = \left[X\right]_\mathfrak{s} + \left[X\right]_\mathfrak{k} \text{ \,\,\,and\,\,\, } g = \left[\,g\,\right]_\mathbb{S} \left[\,g\,\right]_K$$ with $\left[X\right]_\mathfrak{s} \in \mathfrak{s}$, $\left[X\right]_\mathfrak{k} \in \mathfrak{k}$, $\left[\,g\,\right]_\mathbb{S} \in \mathbb{S}$ and $\left[\,g\,\right]_K \in K$. With these notations, we can remark that 
\begin{eqnarray}\label{2}
\tau_g\left(s\right) = \left[\,gs\,\right]_\mathbb{S} \text{\,\,\, and \,\,\,} \left[X\right]_\mathfrak{s} = \left.\frac{d}{dt}\right|_{t=0} \left[\exp\left(t X\right)\right]_\mathbb{S}
\end{eqnarray}
for each $s \in \mathbb{S}$, $g \in G$ and $X \in \fg$.

\begin{rem}\label{mm}
The Lie algebra $\fs$ is endowed with a scalar product $\left(- | -\right)$ induced by the K\"ahlerian structure of \,$\mathbb{S} \simeq \mathbb{D}$. Up to a constant $C_{\mathbb{D}} \in \mathbb{R}$, we have $$\left(\left[X\right]_\mathfrak{s}|\left[Y\right]_\mathfrak{s}\right) = C_{\mathbb{D}} \,\beta_\sigma\left(X, Y\right) = C_{\mathbb{D}} \,\beta\left(X, Y\right)$$ for each $X, Y \in \fp \simeq \fs$; \cite[Ch.\,1, rem.\,1.5.9]{Ko14}. 
In addition, lemma \ref{m}, the $\ad$-invariance of the Killing form $\beta$, and the equality $\left[Y, \left[X\right]_\mathfrak{s}\right] = \left[\left[Y, X\right]\right]_\mathfrak{s}$ for each $X \in \fp$ and $Y \in \fm \subset \fk$, allow us to show that $$X^{\perp\left(\lambda\right)} = \left[\fm, X\right] = \left\{Y \in \fg_\lambda : \left(X\,|\,Y\right) = 0\right\}$$ for all $\lambda \in \Sigma$ and $X \in \fg_\lambda\backslash\left\{0\right\}$.
\end{rem}

\subsection{Pyatetskii-Shapiro decomposition}

The following proposition explicits the so-called \emph{Pyatetskii-Shapiro decomposition} of the Lie group $\mathbb{S} \simeq \mathbb{D}$ into elementary bricks. It is obtained by combining results from the reference \cite[Ch.\,2, \S\,3]{Py69} as well as \cite[Ch.\,1, thm.\,1.125]{Kn02} and \cite[Ch.\,1, lem.\,1.3.10, lem.\,1.4.12 \& prop.\,1.5.10]{Ko14}.

\begin{prop}\label{cool} There exists $n_1, ..., n_r \in \mathbb{N}\backslash\left\{0\right\}$ such that the Lie group $\mathbb{S}$ admits the decomposition $$\mathbb{S} = \left(...\left(\mathbb{S}_r \ltimes \mathbb{S}_{r-1}\right) \ltimes ... \ltimes \mathbb{S}_2\right) \ltimes \mathbb{S}_1$$
where $\mathbb{S}_j$ is a Lie subgroup of $\mathbb{S}$ which is isomorphic to the Iwasawa group of $G_j := SU\left(1, n_j\right)$ for each $1 \leq j \leq r$. In addition, the group $\mathbb{S}_j$ acts simply transitively on the complex unit ball of \,$\mathbb{C}^{n_j}$ and this space admits a structure of symmetric bounded domain with automorphism group $G_j$.
\end{prop}

\noindent In some sense, the complex unit ball of $\mathbb{C}^{N}$ is part of the building blocks of every symmetric bounded domain. We will further explicit our quantization method for this elementary case. 

Originally, this decomposition was written at the infinitesimal level from \cite[Ch.\,2, lem.\,1]{Py69} where the previous proposition finds its root. We can formulate its Lie algebraic version in the following way.

\begin{lem}\label{psds}
The Lie algebra $\fs$ can be decomposed as $$\fs = \left(...\left(\fs_r \ltimes \fs_{r-1}\right) \ltimes ... \ltimes \fs_2\right) \ltimes \fs_1$$ where, for each $1 \leq j \leq r$, the factor $\fs_j$ is a Lie subalgebra of $\fs$ which contains:
\\ \scriptsize $\bullet$ \normalsize \,\,a generator $E_j$ of a one-dimensional ideal of $\mathfrak{s}_j$,
\\ \scriptsize $\bullet$ \normalsize \,\,a vector subspace $V_j \subset \mathfrak{s}_j$ of dimension $2\left(n_j - 1\right) \in \mathbb{N}$ endowed with a symplectic form $\Omega_j \in V_j^\star \otimes V_j^\star$,
\\ \scriptsize $\bullet$ \normalsize \,\,an element $H_j \notin V_j \oplus \mathbb{R} E_j$,
\\ such that $$\fs_j = \mathbb{R} H_j \ltimes \left(V_j \oplus \mathbb{R} E_j\right)$$ with the Lie bracket described by the equalities $$\left[v_j, E_j\right] = 0\text{\hspace{0.2 mm}}, \,\,\, \left[v_j, v_j^\prime\right] = \Omega_j\left(v_j, v_j^\prime\right) E_j \text{\,\,\, and \,\,\,} \left[H_j, v_j + z E_j\right] = v_j + 2 z E_j$$ for all $v_j, v_j^\prime \in V_j$ and $z \in \mathbb{R}$. The Lie algebra structure of $\fs$ satisfies $$\left[X, H_j\right] = \left[X, E_j\right] = 0 \text{\,\,\, and \,\,\,} \ad_X \in \mathfrak{sp}\left(V_j, \Omega_j\right)$$ for each $1 \leq j \leq r-1$ and $X \in \left(\fs_r \ltimes ...\right) \ltimes \fs_{j+1}$.
\end{lem}

\begin{rem} For each $1 \leq j \leq r$, the Lie subalgebra $\fs_j \subset \fs$ is isomorphic to the Lie algebra of the Iwasawa group of $SU\left(1, n_j\right)$; \cite[Ch.\,1, prop.\,1.5.10]{Ko14}.
\end{rem}

\noindent Let's point out that both the number of Lie subalgebras $\fs_j$ and the number of Lie subgroups $\mathbb{S}_j$ in these Pyatetskii-Shapiro decompositions correspond to the rank of the domain $\mathbb{D}$. This fact is not completely obvious in the statement \cite[Ch.\,2, lem.\,1]{Py69} but it can be deduced from the relations $$\fn = \left[\fs, \fs\right] = \bigoplus^{r}_{j = 1} \left(V_j \oplus \mathbb{R} E_j\right) \text{\, \, \, and \, \, \,} \fa \,\simeq\, \bigoplus^{r}_{j = 1} \,\mathbb{R} H_j$$ further in the reference \cite[Ch.\,2, \S\,3]{Py69}.

\section{Intertwining invariant deformation quantizations}

Till the end of this article, we will work through the identification $\mathbb{D} \simeq \mathbb{S}$.

\noindent We now introduce the premises of a strategy leading to a realization of the space of the $G$-invariant star-products on the domain $\mathbb{D}$. This strategy is based on the \emph{retract method} initiated in \cite{B+09} and further extended by Bieliavsky both in for formal and non-formal deformation quantizations; \cite[Ch.\,2, \S\,5\,\&\,8]{Ko14}, \cite{Bi17}. Roughly speaking, in this context, this method can be described by two steps:

\vspace{-5 mm}

\begin{enumerate}[(i)]
\item \textit{computing a set of invariant deformation quantizations on a curvature contraction of $\mathbb{D}$ sharing a common symmetry group with $\mathbb{D}$ ;}
\item \textit{intertwining these deformation quantizations with equivariant operators reversing the contraction process.}
\end{enumerate}

\vspace{-5 mm}

\noindent This approach is intuitively motivated by the fact that it should be easier to compute invariant deformation theory on a curvature contraction of $\mathbb{D}$. Once step \texttt{(i)} is completed, the difficulty consists in reversing the contraction process. In the case of formal deformation quantizations, ie. star-products, the intertwiners are expressed as formal differential operators called \emph{equivalence of invariant star-products}. In the case of non-formal deformation quantizations, intertwiners calculus involves equivalence of Lie group representations.

\subsection{Equivalence of invariant star-products}

Our starting point is the recent memoir \cite{BG15} in which Bieliavsky and Gayral developped a formal and non-formal left-invariant deformation theory on every negatively curved left-invariant K\"ahlerian Lie group. In particular, their work yields an explicit infinite dimensional parameter family of $\mathbb{S}$-invariant star-products on $\mathbb{S}$, each of them underlying a non-formal deformation quantization. With the objective of exploiting this major result, we are going to use well-established properties of star-products in order to transform such $\mathbb{S}$-invariant star-products into $G$-invariant ones.

\begin{df}
Let $G_1$ be a Lie subgroup of $G$. Two $G_1$-invariant star-products $\ast_\nu$ and $\ast^\prime_\nu$ on $\mathbb{D}$ are said to be \emph{$G_1$-equivalent} if there exists a sequence $\left\{T_k : k \in \mathbb{N}\backslash\left\{0\right\}\right\}$ of $\mathbb{C}\llbracket\nu\rrbracket$-linear differential operators on $\mathcal{C}^\infty\left(\mathbb{D}\right)\llbracket\nu\rrbracket$ that vanish on constants, commute with the action $\tau$ of $G_1$ and are such that the operator 
\begin{eqnarray}\label{equiv}
T = \Id + \sum_{k = 1}^\infty \nu^k T_k \text{ \,\,satisfies\,\, } T\left(f_1 \ast_\nu f_2\right) = T\left(f_1\right) \ast^\prime_\nu T\left(f_2\right)
\end{eqnarray}
for each $f_1, f_2 \in \mathcal{C}^\infty\left(\mathbb{D}\right)$. In this case, the operator $T$ is called a \emph{$G_1$-equivalence} and this relation between $\ast_\nu$ and $\ast^\prime_\nu$ is denoted by $\ast^\prime_\nu = T\left(\ast_\nu\right)$. 
\end{df}

\noindent In the present text, if $\ast_\nu$ is a $\mathbb{S}$-invariant star-product on $\mathbb{D}$, the notation $\Op^{\mathbb{S}}\left(\ast_\nu\right)$ will designate the collection of $\mathbb{S}$-equivalences between $\ast_\nu$ and any other $\mathbb{S}$-equivalent $\mathbb{S}$-invariant star-product on $\mathbb{D}$. The following remark from harmonic analysis is quite important in a non-formal perspective.

\begin{rem} Let $T$ be a $\mathbb{S}$-equivalence of star-products of the form (\ref{equiv}). Through the identification $\mathbb{D} \simeq \mathbb{S}$, given that $\tau_s = L_s$ for all $s \in \mathbb{S}$, the operator $T_k$ has to commute with the left-invariant translations on $\mathbb{S}$ for each $k \in \mathbb{N}\backslash\left\{0\right\}$. As a consequence, the $\mathbb{S}$-equivalence $T$ should necessarily be an invertible linear convolution operators on $\mathcal{C}^\infty\left(\mathbb{S}\right)\llbracket\nu\rrbracket$. If $ds$ denotes the left-invariant Haar measure on $\mathbb{S}$, we then have 
\begin{eqnarray}\label{conv} 
T : f \in \mathcal{D}\left(\mathbb{S}\right) \mapsto \left[T\left(f\right) : s_0 \in \mathbb{S} \mapsto \int_{\mathbb{S}} u_T\left(s^{-1} s_0\right) f\left(s\right) ds \right] 
\end{eqnarray}
where $u_T \in \mathcal{D}^\prime\left(\mathbb{S}\right)\llbracket\nu\rrbracket$ is a formal distribution on $\mathbb{S}$ associated with $T$.
\end{rem}

\noindent In view of this remark, it would be legitimate to express some of these $\mathbb{S}$-equivalences within a functional framework allowing to compute explicitly $G$-invariant non-formal deformation quantizations on $\mathbb{D}$ in further work.

\subsection{Classification results}

For every explicit $\mathbb{S}$-invariant star-product $\ast_\nu$ obtained in \cite{BG15}, a natural approach to our quantization problem would be to determine the set of $\mathbb{S}$-equivalences of star-products $T \in \Op^{\mathbb{S}}\left(\ast_\nu\right)$ such that $T\left(\ast_\nu\right)$ is $G$-invariant. 

\noindent Nevertheless, in order to justify such method, we have to prove that every $G$-invariant star-product on $\mathbb{D}$ can be reached in this way. For this, we need the following classification result.

\begin{prop}\label{classif} \emph{\cite[thm.\,4.1]{B+98}}
For every Lie subgroup $G_1 \subset G$, the $G_1$-equivalence classes of $G_1$-invariant star-products on $\mathbb{D}$ are parametrized by the space of formal power series with coefficients in the second cohomology space of the $G_1$-invariant de Rham complex on $\mathbb{D}$.
\end{prop}

\noindent As a consequence of this proposition, the $\mathbb{S}$-equivalence classes of $\mathbb{S}$-invariant star-products on $\mathbb{D} \simeq \mathbb{S}$ are parametrized by the space of formal power series with coefficients in the second Chevalley-Eilenberg cohomology space $H^2_{\CE}\left(\fs\right)$ for the trivial representation of $\fs$ on $\mathbb{C}$. Let's compute explicitly this cohomology space by using the Pyatetskii-Shapiro decomposition of $\fs$. 

\begin{lem}\label{2chc}
In the notations of lemma \ref{psds}, an anti-symmetric bilinear map \,$c : \fs \times \fs \rightarrow \mathbb{C}$\, defines a Chevalley-Eilenberg $2$-cocycle for the trivial representation of $\fs$ on $\mathbb{C}$ if and only if it satisfies the following conditions:

\vspace{-11 mm}

\begin{eqnarray}
\nonumber &\small (i) \normalsize& c\left(v_j, E_j\right) = 0 \text{ \, \,and\, \, } 2\,c\left(v_j, v_j^\prime\right) = \Omega_j\left(v_j, v_j^\prime\right) c\left(H_j, E_j\right) \text{ \,\,\,for each $1 \leq j \leq r$ and $v_j, v_j^\prime \in V_j$;} 
\\ \nonumber &\small (ii) \normalsize& c\left(X, E_j\right) = 0 \text{ \,\, for each $1 \leq j < r$ and $X \in \left(\fs_r \ltimes ...\right) \ltimes \fs_{j+1}$;}
\\ \nonumber &\small (ii') \normalsize& c\left(X, v_j\right) = c\left(H_j, \left[X, v_j\right]\right) \text{ \,\,\,for each $1 \leq j < r$, $X \in \left(\fs_r \ltimes ...\right) \ltimes \fs_{j+1}$ and $v_j \in V_j$;}
\\ \nonumber &\small (iii) \normalsize& c\left(v_k, H_j\right) = 0 \text{ \, \,and\, \, } c\left(E_k, H_j\right) = 0 \text{ \,\, for each $1 \leq j < k \leq r$ and $v_k \in V_k$.}
\end{eqnarray}

\vspace{-5 mm}

\noindent In particular, the data of such a Chevalley-Eilenberg $2$-cocycle $c$ is completely determined by an arbitrary choice of: 
\\ \scriptsize $\bullet$ \normalsize \,\,linear maps \,$c_j : V_j \oplus \mathbb{R} E_j \rightarrow \mathbb{C} : X \mapsto c_j\left(X\right) := c\left(H_j, X\right)$\, for \,$1 \leq j \leq r$;
\\ \scriptsize $\bullet$ \normalsize \,\,constants \,$c_{jk} := c\left(H_j, H_k\right) \in \mathbb{C}$\, for \,$1 \leq j < k \leq r$.
\end{lem}

\begin{proof} 
Let $c : \fs \times \fs \rightarrow \mathbb{C}$ be an anti-symmetric bilinear map. By definition, it is a Chevalley-Eilenberg $2$-cocycle for the trivial representation of $\fs$ on $\mathbb{C}$ if and only if it satisfies
\begin{eqnarray}\nonumber
\delta c \left(X, Y, Z\right) := c\left(\left[X, Y\right], Z\right) + c\left(\left[Y, Z\right], X\right) + c\left(\left[Z, X\right], Y\right) = 0
\end{eqnarray}
for all $X, Y, Z \in \mathfrak{s}$. We are going to use the properties of $c$ and the Lie algebra structure of $\fs$ described in lemma \ref{psds} in order to implement explicitly this condition on $c$. 
We proceed by induction on the rank $r$ of the domain $\mathbb{D} \simeq \mathbb{S}$. 

\vspace{-3 mm}

\noindent \textsc{Initial step.} In the case $r = 1$, we can remark that relation $(i)$ is equivalent to the equations \,$\delta c \left(H_1, E_1, v_1\right) = 0$\, and \,$\delta c \left(H_1, v_1, v^\prime_1\right) = 0$\, for $v_1, v^\prime_1 \in V_1$. An easy computation shows that these equalities implies $\delta c = 0$. As a consequence, the result follows given that relations $(ii)$, $(ii')$ and $(iii)$ are trivially satisfied.

\vspace{-3 mm}

\noindent \textsc{Inductive step.} Let's assume that the statement of the lemma is true for $r = r_0 \in \mathbb{N}\backslash\left\{0\right\}$ and let's prove it for $r = r_0+1$. We set $$\mathfrak{S} := \left(...\left(\fs_{r_0 + 1} \ltimes \fs_{r_0}\right) \ltimes ... \ltimes \fs_3\right) \ltimes \fs_2, \,\,\,\,\, c_\mathfrak{S} := \left.c\right|_{\mathfrak{S} \times \mathfrak{S}} \text{ \,\, and \,\, } c_{\mathfrak{s}_1} := \left.c\right|_{\fs_1 \times \fs_1}.$$ We notice that $c$ is a Chevalley-Eilenberg $2$-cocycle if and only if $c_\mathfrak{S}$ and $c_{\mathfrak{s}_1}$ are Chevalley-Eilenberg $2$-cocycles and $\delta c$ vanishes on $\mathfrak{S} \times \mathfrak{S} \times \fs_1$ and $\mathfrak{S} \times \fs_1 \times \fs_1$. For each $X, Y \in \mathfrak{S}$ and $v_1, v^\prime_1 \in V_1$, the Lie algebra structure of $\fs$ yields the equalities

\vspace{-10.5 mm}

\begin{eqnarray}
\nonumber \text{\hspace{-2.5 mm}} &\text{\small (II)} \normalsize& \delta c\left(H_1, E_1, X\right) = 2\,c\left(E_1, X\right),
\\ \nonumber \text{\hspace{-2.5 mm}} &\text{\small (II')} \normalsize& \delta c\left(H_1, v_1, X\right) = c\left(v_1, X\right) + c\left(H_1, \left[X, v_1\right]\right),
\\ \nonumber \text{\hspace{-2.5 mm}} &\text{\small (0)} \normalsize& \delta c\left(v_1, E_1, X\right) = c\left(\left[X, v_1\right], E_1\right) \text{ \,\,and\,\, } \delta c\left(v_1, v_1^\prime, X\right) = \Omega_1\left(v_1, v_1^\prime\right) c\left(E_1, X\right) \text{ \,as\, $\ad_X \in \mathfrak{sp}\left(V_1, \Omega_1\right)$},
\\ \nonumber \text{\hspace{-2.5 mm}} &\text{\small (III)} \normalsize& \delta c\left(H_1, X, Y\right) = c\left(\left[X, Y\right], H_1\right) \text{ \,\,and\,\, } \delta c\left(E_1, X, Y\right) = c\left(\left[X, Y\right], E_1\right),
\\ \nonumber \text{\hspace{-2.5 mm}} &\text{\small (0')} \normalsize& \delta c\left(v_1, X, Y\right) = c\left(\left[v_1, X\right], Y\right) + c\left(\left[X, Y\right], v_1\right) + c\left(\left[Y, v_1\right], X\right).
\end{eqnarray}

\vspace{-5 mm}

\noindent \scriptsize $\bullet$ \normalsize \,\,\textsc{Necessary condition.} If $c$ is a Chevalley-Eilenberg $2$-cocycle, as $c_\mathfrak{S}$ and $c_{\mathfrak{s}_1}$ satisfy our induction hypothesis, the necessary condition will be proven if we have relations $(ii)$, $(ii')$ and $(iii)$ for $j = 1$. These relations can be respectively deduced from {\small (II)}, {\small (II')} and {\small (III)} given that $$\delta c = 0 \text{ \, and \, } \left[\mathfrak{S}, \mathfrak{S}\right] = \bigoplus^{r_0 + 1}_{j = 2} \left(V_j \oplus \mathbb{R} E_j\right).$$

\noindent \scriptsize $\bullet$ \normalsize \,\,\textsc{Sufficient condition.} Let's assume that relations $(i)$, $(ii)$, $(ii')$ and $(iii)$ are satisfied. Then, by using relation $(ii')$ and the Jacobi identity, the equation {\small (0')} can be written $$\delta c\left(v_1, X, Y\right) = c\left(H_1, \left[\left[v_1, X\right], Y\right]\right) + c\left(H_1, \left[\left[X, Y\right], v_1\right]\right) + c\left(H_1, \left[\left[Y, v_1\right], X\right]\right) = 0$$ for each $X, Y \in \mathfrak{S}$ and $v_1 \in V_1$.
From the combinaison of this last equality with relations {\small (II)}-$(ii)$, {\small (II')}-$(ii')$, {\small (0)}-$(i)$-$(ii)$ and {\small (III)}-$(iii)$, it is clear that $\delta c$ vanishes on $\mathfrak{S} \times \mathfrak{S} \times \fs_1$ and $\mathfrak{S} \times \fs_1 \times \fs_1$. Since the maps $c_\mathfrak{S}$ and $c_{\mathfrak{s}_1}$ are Chevalley-Eilenberg $2$-cocycles by induction hypothesis, the proof of the sufficient condition is complete.
\end{proof}

\vspace{3 mm}

\begin{lem}\label{2cbb}
Let $c : \fs \times \fs \rightarrow \mathbb{C}$ be a Chevalley-Eilenberg $2$-cocycle for the trivial representation of $\fs$ on $\mathbb{C}$. The following assertions are equivalent: 
\\ $(1)$ \,the $2$-cocycle $c$ is a Chevalley-Eilenberg $2$-coboundary;
\\ $(2)$ \,in the notations of lemma \ref{psds}, we have $c\left(H_j, H_k\right) = 0$ for each $1 \leq j \leq r$ and $1 \leq k \leq r$;
\\ $(3)$ \,the $2$-cocycle $c$ vanishes on $\fa \times \fa$.
\end{lem}

\begin{proof} 
We prove separately the implications $(1) \Rightarrow (2) \wedge (3)$, $(2) \Rightarrow (1)$ and $(3) \Rightarrow (1)$.

\vspace{-3.5 mm}

\noindent \scriptsize $\bullet$ \normalsize \,\,By definition, if $c$ is a Chevalley-Eilenberg $2$-coboundary, there exists a linear map $\alpha : \fs \rightarrow \mathbb{C}$ such that $c\left(X, Y\right) = \alpha\left(\left[X, Y\right]\right)$ for each $X, Y \in \fs$. Given that $\mathbb{R} H_1 \oplus ... \oplus \mathbb{R} H_r$ and $\fa$ are two isomorphic abelian Lie subalgebras of $\fs$, the implication $(1) \Rightarrow (2) \wedge (3)$ follows trivially.

\vspace{-3.5 mm}

\noindent \scriptsize $\bullet$ \normalsize \,\,If the $2$-cocycle $c$ satisfies $c\left(H_j, H_k\right) = 0$ for each $j, k \in \left\{1, ..., r\right\}$, we can easily use lemmas \ref{psds} and \ref{2chc} for checking that the linear map $\alpha$, defined on $\fs$ by $$\alpha\left(H_j\right) := 0, \,\,\,\, \alpha\left(v_j\right) := c\left(H_j, v_j\right) \text{ \, and \, } \alpha\left(E_j\right) := \,\frac{c\left(H_j, E_j\right)}{2} \text{ \,\, for all $1 \leq j \leq r$ and $v_j \in V_j$,}$$ is such that $c\left(X, Y\right) = \alpha\left(\left[X, Y\right]\right)$ for each $X, Y \in \fs$. This proves the implication $(2) \Rightarrow (1)$.

\vspace{-3.5 mm}

\noindent \scriptsize $\bullet$ \normalsize \,\,Let's assume that the $2$-cocycle $c$ vanishes on $\fa \times \fa$. In the notations of the root space decomposition of $\fg$, we define the linear map $\alpha^\prime : \fs \rightarrow \mathbb{C}$ by $$\alpha^\prime\left(H_\lambda\right) := 0 \text{ \, and \, } \alpha^\prime\left(X\right) := \,\frac{c\left(H_\lambda, X\right)}{\lambda\left(H_\lambda\right)} \text{ \,\, for all $\lambda \in \Sigma^+$ and $X \in \fg_{\lambda}$.}$$ Then, for each $\lambda, \mu \in \Sigma^+$, $X \in \fg_\lambda$ and $Y \in \fg_\mu$, the properties of $c$ as $2$-cocycle yields
\begin{eqnarray} \nonumber
c\left(H_\lambda, Y\right) \,\,\,=\,\,\, \frac{c\left(H_\lambda, \left[H_\mu, Y\right]\right)}{\mu\left(H_\mu\right)} &=& \frac{c\left(H_\mu, \left[H_\lambda, Y\right] \right) + c\left(Y, \left[H_\mu, H_\lambda\right]\right)}{\mu\left(H_\mu\right)} \\ \nonumber &=& \left(\frac{\mu\left(H_\lambda\right)}{\mu\left(H_\mu\right)}\right) c\left(H_\mu, Y\right) \,\,\,=\,\,\, \mu\left(H_\lambda\right) \alpha^\prime\left(Y\right) \,\,\,=\,\,\, \alpha^\prime\left(\left[H_\lambda, Y\right]\right)
\end{eqnarray}\begin{eqnarray} \nonumber
\text{and \,\,\,} c\left(X, Y\right) \,\,\,=\,\,\, \frac{c\left(\left[H_{\lambda+\mu}, X\right], Y\right) + c\left(X, \left[H_{\lambda+\mu}, Y\right]\right)}{\left(\lambda + \mu\right) \left(H_{\lambda+\mu}\right)} \,\,\,=\,\,\, \frac{c\left(H_{\lambda+\mu}, \left[X, Y\right]\right)}{\left(\lambda + \mu\right)\left(H_{\lambda+\mu}\right)} \,\,\,=\,\,\, \alpha^\prime\left(\left[X, Y\right]\right), \,\,\,\,\,\,\,
\end{eqnarray}
where the last equality comes from the relation $\left[X, Y\right] \in \left[\mathfrak{g}_\lambda, \mathfrak{g}_\mu\right] \subset \mathfrak{g}_{\lambda + \mu}$. As a consequence, we obtain $c\left(X^\prime, Y^\prime\right) = \alpha^\prime\left(\left[X^\prime, Y^\prime\right]\right)$ for each $X^\prime, Y^\prime \in \fs$ and the proof of the lemma is complete.
\end{proof} 

\noindent In the rest of the text, the notation $\left[c\right]$ will refer to the Chevalley-Eilenberg cohomology class of a Chevalley-Eilenberg $2$-cocycle $c$. The following corollary is direct from proposition \ref{classif} and lemmas \ref{2chc} and \ref{2cbb}.

\begin{cor}\label{H2S}
The map $$\left[c\right] \in H^2_{\CE}\left(\fs\right) \mapsto \left(c\left(H_j, H_k\right)\right)_{1 \leq j < k \leq r}$$ induces an isomorphism between $H^2_{\CE}\left(\fs\right)$ and $\mathbb{C}^{\frac{r\left(r-1\right)}{2}}$. In particular, the $\mathbb{S}$-equivalence classes of $\mathbb{S}$-\break invariant star-products on $\mathbb{D}$ are parametrized by $$\mathbb{C}^{\frac{r\left(r-1\right)}{2}}\llbracket\nu\rrbracket.$$
\end{cor}

We are now going to prove that all the $G$-invariant star-products on $\mathbb{D}$ belong to the same $\mathbb{S}$-equivalence class of $\mathbb{S}$-invariant star-products. For the case $r > 1$, this result is very important in order to certify that the natural approach via intertwiners we described above is not so naive and can really be efficient. In what follows, we will denote by $Z^2\left(\mathbb{D}\right)^G$ the set of $G$-invariant closed differential $2$-form on $\mathbb{D}$.

\begin{lem}\label{GclSex}
Every $G$-invariant closed differential $2$-form on $\mathbb{D}$ is the exterior derivative of a $\mathbb{S}$-invariant differential $1$-form on $\mathbb{D}$.
\end{lem}

\begin{proof} It is well known that the data of any $G$-invariant closed differential $2$-form on $\mathbb{S} \simeq \mathbb{D}$ is given by its evaluation at the base point $\Id \in \mathbb{S}$ which defines a Chevalley-Eilenberg $2$-cocycle for the trivial representation of $\fs$ on $\mathbb{C}$. In particular, for such an arbitrary $2$-cocycle $c$, the proof will be complete if we show that $c$ is a Chevalley-Eilenberg $2$-coboundary. Let's notice that our $G$-invariance hypothesis yields
\begin{eqnarray} \nonumber
c\left(\left[\Ad_k\left(X\right)\right]_{\fs}, \left[\Ad_k\left(Y\right)\right]_{\fs}\right) = c\left({\left(\tau_k\right)_\star}_{\Id}\left(X\right), {\left(\tau_k\right)_\star}_{\Id}\left(Y\right)\right) = c\left(X, Y\right)
\end{eqnarray} 
for each $k \in K$ and $X, Y \in \fs$. The infinitesimal version of this last relation can be written
\begin{eqnarray} \label{infinvform}
c\left(\left[\left[Z, X\right]\right]_{\fs}, Y\right) + c\left(X, \left[\left[Z, Y\right]\right]_{\fs}\right) = 0
\end{eqnarray} 
for each $Z \in \fk$ and $X, Y \in \fs$. Let's now consider $\lambda, \mu \in \Sigma^+$ and $X \in \fg_\lambda$. Since $X + \sigma\left(X\right) \in \fk$, we can use (\ref{1}) and (\ref{infinvform}) in order to obtain
\begin{eqnarray}
\nonumber \beta\left(X, \sigma\left(X\right)\right) \,c\left(H_\lambda, H_\mu\right) \,\,\,\,=\,\,\,\, c\left(\left[X, X + \sigma\left(X\right)\right], H_\mu\right) &=& c\left(X, \left[\left[X + \sigma\left(X\right), H_\mu\right]\right]_{\fs}\right) \\ \nonumber &=& c\left(X, \left[\lambda\left(H_\mu\right) \left(\sigma\left(X\right) - X\right)\right]_{\fs}\right) \\ \nonumber &=& c\left(X, -2\text{\hspace{0.2 mm}}\lambda\left(H_\mu\right) X\right) \,\,\,=\,\,\, 0.
\end{eqnarray} 
Therefore, the $2$-cocycle $c$ vanishes on $\fa \times \fa$ and lemma \ref{2cbb} allows us to conclude the proof.
\end{proof} 

\begin{rem}\label{alpha} 
From the previous proof, the evaluation operator of differential forms on $\mathbb{D}$ at the base point $\Id \in \mathbb{S} \simeq \mathbb{D}$ induces a linear isomorphism between $Z^2\left(\mathbb{D}\right)^G$ and the space of linear maps $\alpha : \fn = \left[\fs, \fs\right] \rightarrow \mathbb{C}$ satisfying 
\begin{eqnarray} \label{c}
\alpha\left(\left[\left[\Ad_k\left(X\right)\right]_{\fs}, \left[\Ad_k\left(Y\right)\right]_{\fs}\right]\right) = \alpha\left(\left[X, Y\right]\right)
\end{eqnarray} 
for each $k \in K$ and $X, Y \in \fs$. In particular, for $Z \in \fm \subset \fk$ and $X, Y \in \fs$, such a linear map $\alpha : \fn \rightarrow \mathbb{C}$ has to satisfy the relation $$\alpha\left(\left[Z, \left[X, Y\right]\right]\right) = \alpha\left(\left[\left[Z, X\right], Y\right]\right) + \alpha\left(\left[X, \left[Z, Y\right]\right]\right) = 0,$$ given that $\fm \subset N\left(\fn\right)$ and $\left[\fm, \fa\right] = 0$. As a consequence, we deduce from condition (\ref{c}) and lemma \ref{m} that such a map $\alpha$ has to vanish on every root space $\fg_\lambda$ with $\lambda \in \Sigma^+$ and $\dim\left(\fg_\lambda\right) > 1$.
\end{rem}

\noindent As a K\"ahlerian manifold, the domain $\mathbb{D}$ is naturally endowed with a $G$-invariant symplectic connection $\nabla$. Then, given an arbitrary formal power series\, $\omega_\nu \in Z^2\left(\mathbb{D}\right)^G\llbracket\nu\rrbracket$, the Fedosov's construction of star-products provides us with a $G$-invariant star-product $\ast^{\left(\nabla, \,\omega_\nu\right)}_{\nu}$ on $\mathbb{D}$; \cite{Fe94}. In addition, every $G$-invariant star-product on $\mathbb{D}$ is $G$-equivalent to a Fedosov star-product of this form; \cite[\S\,4, prop.\,4.1]{B+98}. If $G_1$ is a Lie subgroup of $G$, for each $\omega_\nu, \eta_\nu \in Z^2\left(\mathbb{D}\right)^G\llbracket\nu\rrbracket$, the $G$-invariant star-products $\ast^{\left(\nabla, \,\omega_\nu\right)}_{\nu}$ and $\ast^{\left(\nabla, \,\eta_\nu\right)}_{\nu}$ are $G_1$-equivalent if and only if $\omega_\nu - \eta_\nu$ is a formal power series in $\nu$ which coefficients are exterior derivative of $G_1$-invariant differential $1$-forms on $\mathbb{D}$; \cite[\S\,3, thm.\,3.1 \& thm.\,3.2]{B+98}. As a consequence, we get the following proposition from lemma \ref{GclSex}.

\begin{prop}\label{fed0}
Every $G$-invariant star-product on $\mathbb{D}$ is \,$\mathbb{S}$-equivalent to the Fedosov star-product $\ast^{\left(\nabla, 0\right)}_{\nu}$.
\end{prop}

\noindent Among the set of $\mathbb{S}$-invariant star-products explicitly described in the work of Bieliavsky and Gayral \cite{BG15}, let's choose a star-product $\ast^0_\nu$ which is $\mathbb{S}$-equivalent to the Fedosov star-product $\ast^{\left(\nabla, 0\right)}_{\nu}$. Then, the previous proposition allows us to refine a method for realizing the space of $G$-invariant star-products on $\mathbb{D}$ as the description of the set of operators $T \in \Op^{\mathbb{S}}\left(\ast^0_\nu\right)$ such that $T\left(\ast^0_\nu\right)$ is $G$-invariant. Let's point out that it is enough to intertwine the only one initial $\mathbb{S}$-invariant star-product $\ast^0_\nu$ in this way for obtaining the full set of $G$-invariant star-products on $\mathbb{D}$.

Let's now conclude this section by some considerations on the classification of the $G$-invariant star-products on $\mathbb{D}$.

\begin{lem}\label{Gequiv}
The $G$-equivalence classes of $G$-invariant star-products on $\mathbb{D}$ are parametrized by the space $Z^2\left(\mathbb{D}\right)^G\llbracket\nu\rrbracket$.
\end{lem}

\begin{proof} For each $\lambda \in \Sigma^+$ and $X \in \fg_\lambda \backslash\left\{0\right\}$, we have $X + \sigma\left(X\right) \in \fk$ and the properties of the root space decomposition of $\fg$ give \begin{eqnarray} \nonumber
&& X \,\,=\,\, \left[X - \left(\frac{X + \sigma\left(X\right)}{2}\right)\right]_{\fs} \,\,=\,\, \left[\frac{X - \sigma\left(X\right)}{2}\right]_{\fs} \,\,=\,\, \left(\frac{1}{\lambda\left(H_\lambda\right)}\right) \left[\left[H_\lambda, \frac{X + \sigma\left(X\right)}{2}\right]\right]_{\fs}\\ \nonumber &\text{and}& H_\lambda \,\,=\,\, \left(\frac{1}{\beta\left(X, \sigma\left(X\right)\right)}\right) \left[X, \sigma\left(X\right)\right] \,\,=\,\, \left(\frac{1}{\beta\left(X, \sigma\left(X\right)\right)}\right) \left[\left[X, X + \sigma\left(X\right)\right]\right]_{\fs}.
\end{eqnarray}
Therefore, we obtain the relation $\fs = \left\{\left[\left[X, Y\right]\right]_{\fs} : X \in \fs,\, Y \in \fk\right\}$. A similar argument as the one used for the proof of lemma \ref{GclSex} shows that a $G$-invariant differential $1$-form on $\mathbb{D}$ is completely determine by a linear map $\alpha : \fs \rightarrow \mathbb{C}$ which satisfies $\alpha\left(\left[\left[X, Y\right]\right]_{\fs}\right) = 0$ for each $X \in \fs$ and $Y \in \fk$. As a consequence, such a $G$-invariant $1$-form is identically zero and the proof is then complete in view of proposition \ref{classif}. 
\end{proof}

\noindent We can notice that there exist $G$-invariant star-products on $\mathbb{D}$ that are not $G$-equivalent since $\omega$ is a $G$-invariant symplectic form on $\mathbb{D}$. Nevertheless, it happens that $\omega$ is the only one generator of the space $Z^2\left(\mathbb{D}\right)^G$. In fact, let's assume that the domain $\mathbb{D}$ is an irreducible Hermitian symmetric space of non compact type as introduced in the reference \cite[Ch.\,8, \S\,5]{He01}. In the notations of the previous section, this hypothesis implies that the adjoint action of $K$ on $\fp$ is irreducible. If we denote by $\iota$ the vector space isomorphism $\fp \simeq \fs$ which associates $\left[X\right]_\mathfrak{s}$ with $X \in \fp$, it follows that the map 
\begin{eqnarray} \label{action}
K \times \fs \rightarrow \fs : \left(k, X\right) \mapsto \left(\iota \circ \Ad_k \circ\, \iota^{-1}\right)\left(X\right) = \left[\Ad_k\left(X\right)\right]_{\fs}
\end{eqnarray}
defines a irreducible action of $K$ on $\fs$. 
Let's now consider $\upsilon$ a $G$-invariant closed differential $2$-form on $\mathbb{D}$ and the linear map $\alpha : \fn \rightarrow \mathbb{C}$ which is associated with $\upsilon$ via remark \ref{alpha}. As the map $\alpha$ satisfies the equality (\ref{c}) for each $k \in K$, the set $E_{\alpha} := \left\{X \in \fs : \alpha \circ \ad_X = 0\right\}$ defines an invariant subspace of $\fs$ for the action (\ref{action}). Since this action is irreducible, we conclude that either $E_{\alpha} = \left\{0\right\}$ or $E_{\alpha} = \fs$.
Given that $\upsilon$ is non-degenerate if and only if $E_{\alpha} = \left\{0\right\}$, then either $\upsilon$ is a $G$-invariant symplectic form on $\mathbb{D}$ or $\upsilon = 0$.\break As a consequence, the structure theory of irreducible Hermitian symmetric spaces of non compact type allows us to deduce the existence of a constant $c_\alpha \in \mathbb{C}$ and a non-zero element $Z_0$ generating the center of the Lie algebra $\fk$ such that $$\alpha\left(\left[X, Y\right]\right) = c_\alpha \,\beta\left(Z_0, \left[\iota^{-1}\left(X\right), \iota^{-1}\left(Y\right)\right]\right)$$ for all $X, Y \in \fs$; \cite[\S\,1, prop.\,1.1]{BM01}. We then obtain the following proposition.

\begin{prop}\label{irr}
As a Hermitian symmetric space of non compact type, if the symmetric bounded domain $\mathbb{D}$ is irreducible, then the symplectic form $\omega$ generates completely the space $Z^2\left(\mathbb{D}\right)^G$. In this case, the $G$-equivalence classes of $G$-invariant star-products on $\mathbb{D}$ are parametrized by $\mathbb{C}\llbracket\nu\rrbracket$.
\end{prop}

\section{Derivations and quantum moment maps}

Let $\ast_\nu$ be an arbitrary star-product on $\mathbb{D}$. If $\left\{D_k : k \in \mathbb{N}\right\}$ is a sequence of $\mathbb{C}\llbracket\nu\rrbracket$-linear differential operators on $\mathcal{C}^\infty\left(\mathbb{D}\right)\llbracket\nu\rrbracket$ such that the operator 
\begin{eqnarray}\nonumber
D = \sum_{k \in \mathbb{N}} \nu^k D_k \text{ \,\,satisfies\,\, } D\left(f_1 \ast_\nu f_2\right) \,=\, D\left(f_1\right) \ast_\nu f_2 \,+\, f_1 \ast_\nu D\left(f_2\right)
\end{eqnarray}
for each $f_1, f_2 \in \mathcal{C}^\infty\left(\mathbb{D}\right)\llbracket\nu\rrbracket$, then we say that $D$ is a \emph{derivation of $\ast_\nu$}. We define $\Der\left(\ast_\nu\right)$ to be the set of derivations of $\ast_\nu$. 

\begin{rem}\label{10} 
Let $G_1$ be a connected Lie subgroup of $G$ with Lie algebra $\fg_1$. The star-product $\ast_\nu$ satisfies condition (\ref{4}) for each $g \in G_1$ if and only if $X^\star \in \Der\left(\ast_\nu\right)$ for each $X \in \fg_1$. In particular, this last relation is the infinitesimal version of the $G_1$-invariance condition for $\ast_\nu$.
\end{rem}

\noindent As the domain $\mathbb{D}$ is a connected simply connected symplectic manifold, the derivations of $\ast_\nu$ satisfy the following lemma based on the result \cite[\S\,6, lem.\,6.1]{GR03}.

\begin{lem}\label{intder} 
For each $D \in \Der\left(\ast_\nu\right)$, there exists $f_D \in \mathcal{C}^\infty\left(\mathbb{D}\right)\llbracket\nu\rrbracket$ such that $$D = \frac{1}{2 \nu} \left[f_D, -\right]_{\ast_\nu} : f \in \mathcal{C}^\infty\left(\mathbb{D}\right)\llbracket\nu\rrbracket \,\longmapsto\, \frac{1}{2 \nu} \left[f_D, f\right]_{\ast_\nu} := \frac{1}{2 \nu} \left(f_D \ast_\nu f - f \ast_\nu f_D\right).$$ In addition, for all $f \in \mathcal{C}^\infty\left(\mathbb{D}\right)\llbracket\nu\rrbracket$, the operator $D_f := \frac{1}{2 \nu} \left[f, -\right]_{\ast_\nu}$ defines a derivation of $\ast_\nu$ and we have $D_f = 0$ if and only if $f \in \mathbb{C}\llbracket\nu\rrbracket$. 
\end{lem}

\noindent For all $f_1, f_2 \in \mathcal{C}^\infty\left(\mathbb{D}\right)\llbracket\nu\rrbracket$, it is standard to notice that 
\begin{eqnarray}\label{compoder}
\left[\frac{1}{2 \nu} \left[f_1, -\right]_{\ast_\nu},  \frac{1}{2 \nu} \left[f_2, -\right]_{\ast_\nu}\right] =  \frac{1}{2 \nu} \left[ \frac{1}{2 \nu} \left[f_1, f_2\right]_{\ast_\nu}, -\right]_{\ast_\nu}.
\end{eqnarray}

The following result is inspired by \cite[\S\,6]{Xu98} and \cite[Ch.\,2, \S\,3]{Ko14}. 

\begin{lem}\label{5}
Let $D : \fg \rightarrow \Der\left(\ast_\nu\right) : X \mapsto D_X$ be a Lie algebra homomorphism. Then, there exists a unique linear map $$\Lambda : \fg \rightarrow \mathcal{C}^\infty\left(\mathbb{D}\right)\llbracket\nu\rrbracket : X \mapsto \Lambda_X$$ such that 
\begin{eqnarray}\label{7}
D_X = \frac{1}{2 \nu} \left[\Lambda_X, -\right]_{\ast_\nu} \text{\,\, and \,\,\,} \Lambda_{\left[X, Y\right]} = \frac{1}{2 \nu} \left[\Lambda_X, \Lambda_Y\right]_{\ast_\nu}
\end{eqnarray}
for each $X, Y \in \fg$. In addition, the $0^{\text{th}}$ order term of $\Lambda$ in $\nu$ coincides with the moment map $\lambda$ if and only if the $0^{\text{th}}$ order term of $D_X$ in $\nu$ coincides with $X^\star$ for each $X \in \fg$.
\end{lem}

\begin{proof}
The existence of a linear map $\hat{\Lambda} : \fg \rightarrow \mathcal{C}^\infty\left(\mathbb{D}\right)\llbracket\nu\rrbracket : X \mapsto \hat{\Lambda}_X$ such that $$D_X = \frac{1}{2 \nu} \left[\hat{\Lambda}_X, -\right]_{\ast_\nu} \text{ \,for any\, } X \in \fg$$ is clear from the linearity of $D$ and lemma \ref{intder}. Let's consider such a map $\hat{\Lambda}$. As $D$ is a Lie algebra homomorphism, we deduce from relation (\ref{compoder}) and lemma \ref{intder} the existence of an antisymmetric bilinear map $c : \fg \times \fg \rightarrow \mathbb{C}\llbracket\nu\rrbracket$ such that $$\hat{\Lambda}_{[X, Y]} = \frac{1}{2 \nu} \left[\hat{\Lambda}_X, \hat{\Lambda}_Y\right]_{\ast_\nu} +\, c\left(X, Y\right)$$ for each $X, Y \in \fg$. The Jacobi identity allows us to remark that $c$ is a Chevalley-Eilenberg $2$-cocycle for the trivial representation of $\fg$ on $\mathbb{C}$. As the Lie algebra $\fg$ is semi-simple, an application of the Whitehead lemma provides us with a linear map $\alpha : \fg \rightarrow \mathbb{C}\llbracket\nu\rrbracket$ such that $$c\left(X, Y\right) = \alpha\left(\left[X, Y\right]\right) \text{ \,for all\, } X, Y \in \mathfrak{g}.$$ Then, the map $\Lambda := \hat{\Lambda} - \alpha$ satisfies (\ref{7}) for each $X, Y \in \fg$. This proves the existence result of this lemma. In order to prove the unicity of $\Lambda$, let's consider an arbitrary linear map $\Lambda^\prime : \fg \rightarrow \mathcal{C}^\infty\left(\mathbb{D}\right)\llbracket\nu\rrbracket$ such that the analog of conditions (\ref{7}) hold for all $X, Y \in \fg$. In view of lemma \ref{intder}, for each $X \in \fg$, we have $\Lambda_X - \Lambda^\prime_X \in \mathbb{C}\llbracket\nu\rrbracket$, and then
\begin{eqnarray}\label{12}
\Lambda_{\left[X, Y\right]} =\, \frac{1}{2 \nu} \left[\Lambda_X, \Lambda_Y\right]_{\ast_\nu} = \,\frac{1}{2 \nu} \left[\Lambda^\prime_X, \Lambda^\prime_Y\right]_{\ast_\nu} = \Lambda^\prime_{\left[X, Y\right]}
\end{eqnarray}
for all $X, Y \in \fg$. As the Lie algebra $\fg$ is semi-simple, it coincides with its derived algebra $\left[\fg, \fg\right]$ and we get $\Lambda = \Lambda^\prime$ from (\ref{12}). This leads us to the unicity of $\Lambda$. For each $X \in \fg$, let's define $\Lambda^0_X \in \mathcal{C}^\infty\left(\mathbb{D}\right)$ to be the $0^{\text{th}}$ order term of $\Lambda_X$ in $\nu$. From the definition of star-product on a symplectic manifold, we deduce  $$D_X = \left\{\Lambda^0_X, -\right\} + o\left(\nu\right) \text{ \,for all\, } X \in \fg.$$ As a consequence, if the $0^{\text{th}}$ order term of $D_X$ in $\nu$ coincides with $X^\star$ for each $X \in \fg$, the map $$\Lambda^0 : \fg \rightarrow \mathcal{C}^\infty\left(\mathbb{D}\right) : X \mapsto \Lambda^0_X$$ satisfies the same properties as the moment map $\lambda$ and the unicity of such a map implies $\Lambda^0 = \lambda$. Reciprocally, if $\Lambda^0 = \lambda$, then $D_X = \left\{\lambda_X, -\right\} + o\left(\nu\right) = X^\star + o\left(\nu\right)$ for each $X \in \fg$. The proof is complete.
\end{proof}

\noindent We notice that the semi-simplicity of $\fg$ is crucial in the previous statement.

\begin{df}
In the notations of lemma \ref{5}, such a linear map $\Lambda = \lambda + o\left(\nu\right)$ will be called \emph{quantum moment map} associated with $D$.
\end{df}

Similarly to the previous section, we denote by $Z^2\left(\mathbb{D}\right)^{\mathbb{S}}$ the set of $\mathbb{S}$-invariant closed differential $2$-form on $\mathbb{D} \simeq \mathbb{S}$. Let's recall the notation $\nabla$ for the $G$-invariant symplectic connection associated with the K\"ahlerian structure of $\mathbb{D}$. From proposition \ref{fed0}, we know that the $\mathbb{S}$-invariant star-products on $\mathbb{D}$ which are $\mathbb{S}$-equivalent to the Fedosov star-product $\ast^{\left(\nabla, 0\right)}_{\nu}$ play a particular role in our work. In the spirit of this section on quantum moment maps, we are now going to introduce an alternative definition of these star-products. 

\begin{prop}\label{pqmm}
The star-product $\ast_\nu$ is $\mathbb{S}$-invariant and $\mathbb{S}$-equivalent to the Fedosov star-product $\ast^{\left(\nabla, 0\right)}_{\nu}$\break if and only if there exists a linear map $\Lambda : \fs \rightarrow \mathcal{C}^\infty\left(\mathbb{D}\right)\llbracket\nu\rrbracket : X \mapsto \Lambda_X$ such that (\ref{7}) hold with $D_X = X^\star$ for each $X, Y \in \fs$.
\end{prop}

\begin{proof} 
It is clear from remark \ref{10} and lemma \ref{intder} that the $\mathbb{S}$-invariance property of $\ast_\nu$ is expressed both in the necessary condition and in the sufficient condition of this proposition. As a consequence, we can choose an operator $T \in \Op^{\mathbb{S}}\left(\ast_{\nu}\right)$ and a formal power series $\omega_\nu \in Z^2\left(\mathbb{D}\right)^{\mathbb{S}}\llbracket\nu\rrbracket$ such that the star-product $\ast^{\left(\nabla, \,\omega_\nu\right)}_{\nu}$ provided by the Fedosov's construction \cite{Fe94} satisfies the relation \,$T\left(\ast_\nu\right) = \ast^{\left(\nabla, \,\omega_\nu\right)}_{\nu}$; \cite[\S\,4, prop.\,4.1]{B+98}. In view of remark \ref{10}, the vector field $X^\star$ is a derivation of $\ast^{\left(\nabla, \,\omega_\nu\right)}_{\nu}$ for each $X \in \fs$. The combinaison of this fact with well known results on quantum moment maps for Fedosov star-products produces a linear map $\hat{\Lambda} : \fs \rightarrow \mathcal{C}^\infty\left(\mathbb{D}\right)\llbracket\nu\rrbracket : X \mapsto \hat{\Lambda}_X$ such that 
\begin{eqnarray}\label{condi}
\iota_{X^\star}\left(\omega + \nu \,\omega_\nu\right) = - d\hat{\Lambda}_X  \text{\,\, and \,\,} X^\star = \frac{1}{2 \nu} \left[\hat{\Lambda}_X, -\right]_{\ast^{\left(\nabla, \,\omega_\nu\right)}_{\nu}}
\end{eqnarray}
for each $X \in \fs$; \cite[\S\,7, thm.\,7.2]{GR03}. After these preliminary considerations, let's prove separately the necessary condition and the sufficient condition of this proposition.

\vspace{-3.5 mm}

\noindent \scriptsize $\bullet$ \normalsize \,\,\textsc{Necessary condition.} In view of our hypothesis, we can choose $\omega_\nu = 0$. Therefore, the definition of the moment map $\lambda$, relation (\ref{condi}) and lemma \ref{intder} and allow us to deduce that $\hat{\Lambda}_X - \lambda_X \in \mathbb{C}\llbracket\nu\rrbracket$ and $X^\star \,=\, \frac{1}{2 \nu} \left[\lambda_X, -\right]_{\ast^{\left(\nabla, \,\omega_\nu\right)}_{\nu}}$ for each $X \in \fs$. In particular, as the operator $T$ is a $\mathbb{S}$-equivalence, we obtain 
\begin{eqnarray} \nonumber && X^\star \,=\, T^{-1} \circ \left(\frac{1}{2 \nu} \left[\lambda_X, -\right]_{\ast^{\left(\nabla, \,\omega_\nu\right)}_{\nu}}\right) \circ \,T \,=\, \frac{1}{2 \nu} \left[T^{-1}\left(\lambda_X\right), -\right]_{\ast_{\nu}} \\  \nonumber &\text{and}& \frac{1}{2 \nu} \left[T^{-1}\left(\lambda_X\right), T^{-1}\left(\lambda_Y\right)\right]_{\ast_\nu} \,=\, T^{-1}\left(\frac{1}{2 \nu} \left[\lambda_X, \lambda_Y\right]_{\ast^{\left(\nabla, \,\omega_\nu\right)}_{\nu}}\right) \,=\, T^{-1}\left(X^\star\left(\lambda_Y\right)\right) \,=\, T^{-1}\left(\lambda_{\left[X, Y\right]}\right).
\end{eqnarray}
for each $X, Y \in \fs$. As a consequence, the linear map \,$T^{-1} \circ \lambda$\, satisfies (\ref{7}) with $D_X = X^\star$ for each $X, Y \in \fs$. This proves the necessary condition.

\vspace{-3.5 mm}

\noindent \scriptsize $\bullet$ \normalsize \,\,\textsc{Sufficient condition.} Let $\Lambda : \fs \rightarrow \mathcal{C}^\infty\left(\mathbb{D}\right)\llbracket\nu\rrbracket : X \mapsto \Lambda_X$ be a linear map such that (\ref{7}) hold with $D_X = X^\star$ for each $X, Y \in \fs$. By using the properties of $T \in \Op^{\mathbb{S}}\left(\ast_{\nu}\right)$ and the definition of $\hat{\Lambda}$, we get
$$\frac{1}{2 \nu} \left[\hat{\Lambda}_X, -\right]_{\ast^{\left(\nabla, \,\omega_\nu\right)}_{\nu}} \,=\, X^\star \,=\, T \circ X^\star \circ T^{-1} \,=\, T \circ \left(\frac{1}{2 \nu} \left[\Lambda_X, -\right]_{\ast_\nu}\right) \circ \,T^{-1} \,=\, \frac{1}{2 \nu} \left[T\left(\Lambda_X\right), -\right]_{\ast^{\left(\nabla, \,\omega_\nu\right)}_{\nu}}$$ for all $X \in \fs$. In particular, we deduce from lemma \ref{intder} and relation (\ref{condi}) that $\hat{\Lambda}_X - T\left(\Lambda_X\right) \in \mathbb{C}\llbracket\nu\rrbracket$ and then $\iota_{X^\star}\left(\omega + \nu \,\omega_\nu\right) = - d \left(T\left(\Lambda_X\right)\right)$ for all $X \in \fs$. As a result, for each $X, Y \in \fs$, we have 
\begin{eqnarray} \nonumber \left(\omega + \nu \,\omega_\nu\right)\left(X^\star, Y^\star\right) \,\,\,=\,\,\, - d \left(T\left(\Lambda_X\right)\right)\left(Y^\star\right) &=& - \frac{1}{2 \nu} \left[T\left(\Lambda_Y\right), T\left(\Lambda_X\right)\right]_{\ast^{\left(\nabla, \,\omega_\nu\right)}_{\nu}} \\ \nonumber &=& T\left(\frac{1}{2 \nu} \left[\Lambda_X, \Lambda_Y\right]_{\ast_\nu}\right) \,\,\,=\,\,\, T\left(\Lambda_{\left[X, Y\right]}\right).
\end{eqnarray} 
The evaluation of $\omega + \nu \,\omega_\nu$ at the base point $\Id \in \mathbb{S} \simeq \mathbb{D}$ defines a formal power series $c$ in $\nu$ which coefficients are Chevalley-Eilenberg $2$-cocycles for the trivial representation of $\fs$ on $\mathbb{C}$. As we have the equality \,$c\left(X, Y\right) = T\left(\Lambda_{\left[X, Y\right]}\right)\left(\Id\right)$\, for each $X, Y \in \fs$, it is clear that these coefficients are Chevalley-Eilenberg $2$-coboundaries. It follows that $\omega_\nu$ is a formal power series in $\nu$ which coefficients are exterior derivative of $\mathbb{S}$-invariant differential $1$-forms on $\mathbb{D}$. The star-product $\ast^{\left(\nabla, \,\omega_\nu\right)}_{\nu}$ is then $\mathbb{S}$-equivalent to $\ast^{\left(\nabla, 0\right)}_{\nu}$ in virtue of the theorem \cite[\S\,3, thm.\,3.1 \& thm.\,3.2]{B+98} mentioned in the previous section. This concludes the proof of the sufficient condition.
\end{proof}

\noindent Let's point out that the key of this proof lies in the fact that we can choose the map $\hat{\Lambda}$ such that $$T^{-1}\left(\hat{\Lambda}_{\left[X, Y\right]}\right) = \frac{1}{2 \nu} \left[T^{-1}\left(\hat{\Lambda}_X\right), T^{-1}\left(\hat{\Lambda}_Y\right)\right]_{\ast_\nu}$$ for each $X, Y \in \fs$ if and only if $\ast_\nu$ is $\mathbb{S}$-equivalent to $\ast^{\left(\nabla, 0\right)}_{\nu}$. The non-triviality of this statement comes from the fact that the Lie algebra $\fs$ is not semi-simple.

\begin{rem}\label{rem}
In the context of proposition \ref{pqmm}, we deduce from lemma \ref{intder} and relation (\ref{12}) that the map $\Lambda_X$ is uniquely determined for $X \in \fn = \left[\fs, \fs\right]$ and uniquely determined up to a formal constant for $X \in \fa$. Moreover, the same argument as the one used in the proof of lemma \ref{5} shows that $\Lambda = \lambda + o\left(\nu\right)$.
\end{rem}

Let's conclude this section by introducing a particular case of proposition \ref{pqmm}.

\begin{df}\label{covariance}
We say that the star-product $\ast_\nu$ is \emph{$\fs$-covariant} if the equation $$\frac{1}{2 \nu} \left[\lambda_X, \lambda_Y\right]_{\ast_\nu} = \lambda_{\left[X, Y\right]}$$ is satisfied for each $X, Y \in \fs$.
\end{df}

\begin{prop}\label{scov}
The star-product $\ast_\nu$ is \,$\mathbb{S}$-invariant, $\fs$-covariant and $\mathbb{S}$-equivalent to the Fedosov star-product $\ast^{\left(\nabla, 0\right)}_{\nu}$ if and only if 
\,$X^\star = \left\{\lambda_X, -\right\} = \frac{1}{2 \nu} \left[\lambda_X, -\right]_{\ast_\nu}$ for each $X \in \fs$.
\end{prop}

\begin{proof}
The sufficient condition is clear in view of properties of the moment map $\lambda$, definition \ref{covariance} and proposition \ref{pqmm}. Let's assume that $\ast_\nu$ is \,$\mathbb{S}$-invariant, $\fs$-covariant and $\mathbb{S}$-equivalent to $\ast^{\left(\nabla, 0\right)}_{\nu}$ and let's prove the necessary condition. In virtue of proposition \ref{pqmm}, we consider a linear map $\Lambda : \fs \rightarrow \mathcal{C}^\infty\left(\mathbb{D}\right)\llbracket\nu\rrbracket : X \mapsto \Lambda_X$ satisfying (\ref{7}) with $D_X = X^\star$ for each $X, Y \in \fs$. For each $X \in \fs$, we set $$\Lambda^\prime_X := \frac{1}{\nu} \left(\Lambda_X - \lambda_X\right) \in \mathcal{C}^\infty\left(\mathbb{D}\right)\llbracket\nu\rrbracket.$$ 
For all $X, Y \in \fs$, our hypotheses and the properties of the moment map $\lambda$ allow us to deduce the equality $$\frac{1}{2 \nu} \left[\nu \Lambda^\prime_X, \lambda_Y\right]_{\ast_\nu} \,=\, \frac{1}{2 \nu} \left[\Lambda_X - \lambda_X, \lambda_Y\right]_{\ast_\nu} \,=\, X^\star\left(\lambda_Y\right) - \lambda_{\left[X, Y\right]} = 0.$$ As a consequence, for all $X, Y \in \fs$, we obtain $$\frac{1}{2 \nu} \left[\nu \Lambda^\prime_X, \nu \Lambda^\prime_Y\right]_{\ast_\nu} \,=\, \frac{1}{2 \nu} \left[\Lambda_X, \Lambda_Y\right]_{\ast_\nu} - \,\frac{1}{2 \nu} \left[\lambda_X, \lambda_Y\right]_{\ast_\nu} \,=\, \Lambda_{\left[X, Y\right]} - \lambda_{\left[X, Y\right]} \,=\, \nu \Lambda^\prime_{\left[X, Y\right]}.$$ 
For all $X \in \fn$, by using the properties of root space decomposition of $\mathfrak{g}$ and a recursive argument on the order terms of $\Lambda^\prime$ in $\nu$ in the relation $\frac{1}{2} \left[\Lambda^\prime_X, \Lambda^\prime_H\right]_{\ast_\nu} =  \Lambda^\prime_{\left[X, H\right]}$ for $H \in \fa$, we deduce easily the equality $\Lambda^\prime_X = 0$. Therefore, as $\left[\fs, \fs\right] = \fn$, we get $$X^\star\left(\nu \Lambda^\prime_H\right) \,=\, X^\star\left(\Lambda_H - \lambda_H\right) \,=\, \frac{1}{2 \nu} \left[\Lambda_X, \Lambda_H\right]_{\ast_\nu} - \,\lambda_{\left[X, H\right]} \,=\, \nu \Lambda^\prime_{\left[X, H\right]} \,=\, 0$$ for all $X \in \fs$ and $HÊ\in \fa$. It follows that $\Lambda^\prime_H \in \mathbb{C}\llbracket\nu\rrbracket$ for each $H \in \fa$. As a conclusion, we have $\frac{1}{2 \nu} \left[\Lambda^\prime_X, -\right]_{\ast_\nu} = 0$ for each $X \in \fs$. This completes the proof of the proposition.
\end{proof}

\section{Deformation quantization method}

As suggested above in view of reference \cite{BG15} and proposition \ref{fed0}, we fix an explicit $\mathbb{S}$-invariant star-product $\ast_\nu$ on $\mathbb{D}$ which is $\mathbb{S}$-equivalent to the Fedosov star-product $\ast^{\left(\nabla, 0\right)}_{\nu}$ where $\nabla$ is the $G$-invariant symplectic connection associated to the K\"ahlerian structure of $\mathbb{D}$. The data of a $G$-invariant star-products on $\mathbb{D}$ is then equivalent to the data of an invertible linear convolution operator $T \in \Op^{\mathbb{S}}\left(\ast_\nu\right)$ of the form (\ref{conv}) such that $T\left(\ast_\nu\right)$ is $G$-invariant. We will work through the identification $\mathbb{D} \simeq \mathbb{S}$. 

\subsection{Equivariant automorphisms of star-products}

In this subsection, we develop some basic results about the following class of operators $$\Aut\left(\ast_\nu\right) := \left\{T \in \Op^{\mathbb{S}}\left(\ast_\nu\right) : T\left(\ast_\nu\right) = \ast_\nu\right\}.$$ A element of $\Aut\left(\ast_\nu\right)$ we be called \emph{$\mathbb{S}$-automorphism of $\ast_\nu$}. 

\begin{lem}\label{31}
For each operators $T, T^\prime \in \op$ such that $T\left(\ast_\nu\right)$ and $T^\prime\left(\ast_\nu\right)$ are two $G$-equivalent $G$-invariant star-products on $\mathbb{D}$, there exist $S \in \Aut\left(\ast_\nu\right)$ and \,$U$ a $G$-equivalence of $G$-invariant star-products on $\mathbb{D}$ satisfying $$T^\prime = U \circ T \circ S.$$
\end{lem}

\noindent This lemma is direct. In fact, if $U$ denotes a $G$-equivalence of $G$-invariant star-products on $\mathbb{S}$ such that $\left(U \circ T\right)\left(\ast_\nu\right) = T^\prime\left(\ast_\nu\right)$, we can set $S \,:=\, \left(U \circ T\right)^{-1} \circ \,T^\prime \,\in\, \Aut\left(\ast_\nu\right).$

\noindent Let $S$ be a $\mathbb{S}$-automorphism of $\ast_\nu$. Given that $\mathbb{S} \simeq \mathbb{D}$ is a connected simply connected Lie group, there exists $\Lambda \in \mathcal{C}^\infty\left(\mathbb{S}\right)\llbracket\nu\rrbracket$ such that
\begin{eqnarray}\label{11}
S = \exp\left(\frac{1}{2 \nu} \left[\nu \Lambda, -\right]_{\ast_\nu}\right) ;
\end{eqnarray}
\cite[\S\,4, prop.\,23]{Gu11}. Let's consider $X \in \fs$ and $\Lambda_X \in \mathcal{C}^\infty\left(\mathbb{S}\right)\llbracket\nu\rrbracket$ such that $\frac{1}{2 \nu} \left[\Lambda_X, -\right]_{\ast_\nu} = X^\star \in \Der\left(\ast_\nu\right)$. As $S$ commutes with the left translations on $\mathbb{S}$, we have $$0 \,=\, \left[X^\star, \frac{1}{2 \nu} \left[\nu \Lambda, -\right]_{\ast_\nu}\right] \,=\, \frac{1}{2 \nu} \left[\frac{1}{2 \nu} \left[\Lambda_X, \nu \Lambda\right]_{\ast_\nu}, -\right]_{\ast_\nu} =\, \frac{1}{2 \nu} \left[\nu X^\star\left(\Lambda\right), -\right]_{\ast_\nu}$$ where relation (\ref{compoder}) is used in the second equality. From lemma \ref{intder}, we deduce that $X^\star\left(\Lambda\right)$ lies in $\mathbb{C}\llbracket\nu\rrbracket$. In particular, as $X$ is arbitrary, we have $X^\star\left(\Lambda\right) = 0$ for each $X \in \left[\mathfrak{s}, \mathfrak{s}\right] = \mathfrak{n}$. In addition, by using the relation $\Ad_A\left(\fn\right) = \fn$, we get $$0 = \left(\Ad_a\left(X\right)\right)^\star_{an}\left(\Lambda\right) = X^\star_n\left(\tau^\star_a \Lambda\right)$$ for each $a \in A$, $n \in N$ and $X \in \mathfrak{n}$. This leads us to the following proposition.

\begin{lem}\label{17}
Each $S \in \Aut\left(\ast_\nu\right)$ is of the form (\ref{11}) where $\Lambda \in \mathcal{C}^\infty\left(\mathbb{S}\right)\llbracket\nu\rrbracket$ satisfies $$\Lambda\left(an\right) = \Lambda\left(a\right) \text{\, and \,\,} \gamma^S_H := H^\star\left(\Lambda\right) \in \mathbb{C}\llbracket\nu\rrbracket$$ for each $a \in A$, $n \in N$ and $H \in \fa$. In particular, the linear map \,$\gamma^S : H \in \fa \mapsto \gamma^S_H$\, encodes univocally the data of $S$ and the space $\Aut\left(\ast_\nu\right)$ is parametrized by $\mathbb{C}^r\llbracket\nu\rrbracket$.
\end{lem}

\noindent Let's point out that the second assertion follows from the equality 
\begin{eqnarray}\label{actautS}
S\left(\Lambda_H\right) = \Lambda_H - \nu \,\gamma^S_H
\end{eqnarray}
which is valid for each $S \in \Aut\left(\ast_\nu\right)$, $H \in \fa$ and $\Lambda_H \in \mathcal{C}^\infty\left(\mathbb{S}\right)\llbracket\nu\rrbracket$ such that $\frac{1}{2 \nu} \left[\Lambda_H, -\right]_{\ast_\nu} = H^\star$. 


\subsection{Intertwiners and quantum moment maps}

For any operator $T \in \op$ such that $T\left(\ast_\nu\right)$ is a $G$-invariant star-product on $\mathbb{D}$, let's define the map $$D^T : X \in \fg \,\mapsto\, D^T_X := T^{-1} \circ X^\star \circ T.$$ We deduce from remark \ref{10} that $D^T$ is a Lie algebra homomorphism of which the image lies in $\Der\left(\ast_\nu\right)$; \cite[Ch.\,2, lem.\,2.3.3]{Ko14}. As the $0^{\text{th}}$ order term of $T$ in $\nu$ coincides with $\Id$, we have $$D^T_X = X^\star + o\left(\nu\right) \text{ \,for any\, } X \in \fg.$$ As a consequence, lemma \ref{5} provides us with a unique quantum moment map $\Lambda^T = \lambda + o\left(\nu\right)$ associated with $D^T$. The data of this quantum moment map is equivalent to the data of $D^T$.

\begin{prop}\label{32}
The derivation $D^T_X$ does not depend on $T$ for $X \in \fs \oplus \fm$ and it coincides with $X^\star$ if $X \in \fs$. In particular, for each operators $T, T^\prime \in \op$ such that $T\left(\ast_\nu\right)$ and $T^\prime\left(\ast_\nu\right)$ are $G$-invariant star-products on $\mathbb{D}$, we have $$\Lambda^T_X - \Lambda^{T^\prime}_X \in \nu\, \mathbb{C}\llbracket\nu\rrbracket \text{ \,when $X \in \fa \,\oplus\, Z\left(\fm\right)$ \, \, and \, \, } \Lambda^T_X = \Lambda^{T^\prime}_X \text{ \,when $X \in \fn \,\oplus\, \left[\fm, \fm\right]$}.$$
\end{prop}

\begin{proof}
Let's consider $T, T^\prime \in \op$ such that $T\left(\ast_\nu\right)$ and $T^\prime\left(\ast_\nu\right)$ are $G$-invariant star-products on $\mathbb{D} \simeq \mathbb{S}$. As the operators $T$ and $T^\prime$ commute with the left translations on $\mathbb{S}$, it is clear that $$\frac{1}{2 \nu} \left[\Lambda^T_X, -\right]_{\ast_\nu} = D^T_X = X^\star = D^{T^\prime}_X = \frac{1}{2 \nu} \left[\Lambda^{T^\prime}_X, -\right]_{\ast_\nu}$$ for each $X \in \fs$. Then, lemma \ref{intder} and a similar argument to (\ref{12}) give us $$\Lambda^T_X = \Lambda^{T^\prime}_X \text{ \,for each } X \in \fn = \left[\fs, \fs\right] \text{ \,\, and \,\, } \Lambda^T_H - \Lambda^{T^\prime}_H \in \nu\, \mathbb{C}\llbracket\nu\rrbracket \text{ \,for each } H \in \fa.$$ Let's consider $Y \in \fm$. If we combine this result and the inclusion $\fm \subset N\left(\fn\right) \cap \fg_0$ with the quantum moment map property of $\Lambda^T$ and $\Lambda^{T^\prime}$, we obtain $$X^\star\left(\Lambda^T_Y - \Lambda^{T^\prime}_Y\right) = \Lambda^T_{\left[X, Y\right]} - \Lambda^{T^\prime}_{\left[X, Y\right]} = 0$$ for each $X \in \fs$. As $\left\{X^\star_s : X \in \fs\right\}$ spans $T_s\left(\mathbb{S}\right)$ for $s \in \mathbb{S}$, these last relations give us $\Lambda^T_Y - \Lambda^{T^\prime}_Y \in \nu\, \mathbb{C}\llbracket\nu\rrbracket$. In particular, we deduce from lemma \ref{intder} that $D^T_Y = D^{T^\prime}_Y$. Equality (\ref{1c}) and the adaptation of expression (\ref{12}) in this framework allow us to conclude the proof.
\end{proof}

\begin{rem}
The hypothesis made on the star-product $\ast_\nu$ are crucial. We remark that the existence of such a map $\left.\Lambda^T\right|_{\fs} = \lambda + o\left(\nu\right)$ uniquely determined on $\fn$ was already mentioned in proposition \ref{pqmm}, 
\end{rem}

\noindent Let's notice that the previous proposition is specific to the Lie subalgebra $\fs \oplus \fm \subset \fg$ as it coincides with $N\left(\fn\right)$. It shows that the derivation $D^T_X$ depends only on the star-product $\ast_\nu$ for $X \in \fs \oplus \fm$. In particular, if this star-product $\ast_\nu$ is $\fs$-covariant, we have the following result.

\begin{cor}\label{corcov}
The star-product $\ast_\nu$ is $\fs$-covariant if and only if $$\Lambda^T_X = \lambda_X \text{\, \, and \, \,} \Lambda^T_H - \lambda_H \in \nu\, \mathbb{C}\llbracket\nu\rrbracket$$ for each $X \in \fn$, $H \in \fa$ and $T \in \op$ such that $T\left(\ast_\nu\right)$ is $G$-invariant. In this case, if \,$T \in \op$ is any operator such that $T\left(\ast_\nu\right)$ is $G$-invariant, we have $\Lambda^T_Y - \lambda_Y \in \nu\, \mathbb{C}\llbracket\nu\rrbracket$ for each $Y \in \fm$.
\end{cor}

\noindent The first part of this corollary is clear in view of propositions \ref{scov} and \ref{32}. The proof of the second assertion follows from a similar argument to the one use in the proof of proposition \ref{32}. In fact, we have $$X^\star\left(\Lambda^T_Y\right) =  \frac{1}{2 \nu} \left[\Lambda^{T}_X, \Lambda^T_Y\right]_{\ast_\nu} = \Lambda^T_{\left[X, Y\right]} = \lambda_{\left[X, Y\right]} \in \mathcal{C}^\infty\left(\mathbb{S}\right)$$ for all $X \in \fs$, $Y \in \fm$ and $T \in \op$ such that $T\left(\ast_\nu\right)$ is $G$-invariant. The next results follow directly from remark \ref{10} and proposition \ref{32}.

\begin{cor}\label{corsm}
Let's denote by $M \subset K$ the connected Lie subgroup of $G$ which Lie algebra $\mathfrak{m} \subset \fk$. 
Then, all the $G$-invariant star-products on $\mathbb{S}$ which are $\mathbb{S}$-equivalent to $\ast_\nu$ are $\mathbb{S} M$-equivalent. In addition, the following assertions are equivalent:
\newline $\bullet$ \,\,there exists $T \in \op$ such that $T\left(\ast_\nu\right)$ is $G$-invariant and such that $D^T_Y = Y^\star$ for each $Y \in \fm$;
\newline $\bullet$ \,\,for each $T \in \op$ such that $T\left(\ast_\nu\right)$ is $G$-invariant and for each $Y \in \fm$, we have $D^T_Y = Y^\star$;
\newline $\bullet$ \,\,the star-product $\ast_\nu$ is $\mathbb{S} M$-invariant and there exists $T \in \op$ such that $T\left(\ast_\nu\right)$ is $G$-invariant and $\mathbb{S} M$-equivalent to $\ast_\nu$.
\end{cor}

\subsection{Retractable homomorphisms}

We are now going to look at the considerations of the previous subsection from a more general point of view. Let's denote by $\mathbb{H}\left(\ast_\nu\right)$ the space of Lie algebra homomorphisms of the form $$D : \fg \rightarrow \Der\left(\ast_\nu\right) : X \mapsto D_X = X^\star + o\left(\nu\right)$$ such that $D_X = X^\star$ for each $X \in \fs$. From lemma \ref{5}, it is clear that it can be identified to the space of quantum moment maps associated with elements of $\mathbb{H}\left(\ast_\nu\right)$.

\begin{rem}\label{33}
Let $\Lambda$ and $\Lambda^\prime$ be the quantum moment maps associated with $D\in \mathbb{H}\left(\ast_\nu\right)$ and $D^\prime \in \mathbb{H}\left(\ast_\nu\right)$ respectively. An obvious adaptation of the proof of proposition \ref{32} shows that $\Lambda_X - \Lambda^\prime_X \in \nu\, \mathbb{C}\llbracket\nu\rrbracket$ for each $X \in \fs \oplus \fm$. In addition, given that $\Lambda_{\sigma\left(Y\right)}$ satisfies the equations $$X^\star\left(\Lambda_{\sigma\left(Y\right)}\right) = \Lambda_{\left[X, \sigma\left(Y\right)\right]_{\fs \oplus \fm}} + \Lambda_{\left[X, \sigma\left(Y\right)\right]_{\sigma\left(\fn\right)}} \text{\, \,with \,} \left[X, \sigma\left(Y\right)\right]_{\fs \oplus \fm} \in \fs \oplus \fm \text{ \, and \, } \left[X, \sigma\left(Y\right)\right]_{\sigma\left(\fn\right)} \in \sigma\left(\fn\right)$$ for all $X \in \fs$ and $Y \in \fn$, we deduce that the space $\mathbb{H}\left(\ast_\nu\right)$ has a structure of finite dimensional vector space over the field $\mathbb{C}\llbracket\nu\rrbracket$.
\end{rem} 

\begin{lem}\label{20}
Let's consider $D \in \mathbb{H}\left(\ast_\nu\right)$ and $S \in \Aut\left(\ast_\nu\right)$. Then we have $$D^S := S^{-1} \circ D \circ S \,\in\, \mathbb{H}\left(\ast_\nu\right).$$ In addition, the homomorphisms $D$ and $D^S$ coincide if and only if $S = \Id$.
\end{lem}

\begin{proof}
The first assertion is clear as $S = \Id + o\left(\nu\right) \in \Aut\left(\ast_\nu\right)$ commutes with the left translations on $\mathbb{S}$; \cite[Ch.\,2, lem.\,2.3.3]{Ko14}. Let $\Lambda$ be the quantum moment map associated with $D$ and let's consider $X \in \fg_\lambda$ for $\lambda \in \Sigma^+$. As $S$ is a $\mathbb{S}$-automorphism of $\ast_\nu$, we have
$$\left(D_{\sigma\left(X\right)} - D^S_{\sigma\left(X\right)}\right)\left(\Lambda_X\right) \,\,\,=\,\,\, \frac{1}{2 \nu} \left[\Lambda_{\sigma\left(X\right)} - S^{-1}\left(\Lambda_{\sigma\left(X\right)}\right), \Lambda_X\right]_{\ast_\nu} \,\,=\,\, - \,X^\star\left(\Lambda_{\sigma\left(X\right)} - S^{-1}\left(\Lambda_{\sigma\left(X\right)}\right)\right).$$
Lemma \ref{17} allows us to choose $\Lambda^\prime \in \mathcal{C}^\infty\left(\mathbb{S}\right)\llbracket\nu\rrbracket$ which is induced by the data of $S^{-1}$ through the relation $$S^{-1} = \exp\left(\frac{1}{2 \nu} \left[\nu \Lambda^\prime, -\right]_{\ast_\nu}\right).$$ In particular, we have $X^\star\left(\Lambda^\prime\right) = 0$ and $\gamma^{S^{-1}}_H := H^\star\left(\Lambda^\prime\right) \in \mathbb{C}\llbracket\nu\rrbracket$ for each $H \in \fa$. By combining these relations with the Jacobi identity and the expression (\ref{1}), we obtain
\begin{eqnarray}\nonumber
X^\star\left(\frac{1}{2 \nu} \left[\nu \Lambda^\prime, \Lambda_{\sigma\left(X\right)}\right]_{\ast_\nu}\right) &=& - \frac{1}{2 \nu} \left[\nu \Lambda^\prime, \frac{1}{2 \nu} \left[\Lambda_{\sigma\left(X\right)}, \Lambda_{X}\right]_{\ast_\nu} \right]_{\ast_\nu} - \frac{1}{2 \nu} \left[\Lambda_{\sigma\left(X\right)}, \frac{1}{2 \nu} \left[\Lambda_{X}, \nu \Lambda^\prime\right]_{\ast_\nu} \right]_{\ast_\nu} \\Ê\nonumber &=& - \frac{1}{2 \nu} \left[\nu \Lambda^\prime, \Lambda_{\left[\sigma\left(X\right), X\right]}\right]_{\ast_\nu} - \frac{1}{2 \nu} \left[\Lambda_{\sigma\left(X\right)}, \nu X^\star\left(\Lambda^\prime\right)\right]_{\ast_\nu} \\Ê\nonumber &=& - \,\nu\, \beta\left(X, \sigma\left(X\right)\right) \gamma^{S^{-1}}_{H_\lambda} \in \mathbb{C}\llbracket\nu\rrbracket.
\end{eqnarray}
An inductive approach shows that $$X^\star\left(\left(\frac{1}{2 \nu} \left[\nu \Lambda^\prime, -\right]_{\ast_\nu}\right)^k \left(\Lambda_{\sigma\left(X\right)}\right)\right) = 0$$ for each integer $k > 1$. As a consequence, we get $$\left(D_{\sigma\left(X\right)} - D^S_{\sigma\left(X\right)}\right)\left(\Lambda_X\right) = - \,\nu \,\beta\left(X, \sigma\left(X\right)\right) \gamma^{S^{-1}}_{H_\lambda}.$$ Since $\left\{H_\lambda : \lambda \in \Sigma^+\right\}$ spans $\fa$ and $\beta\left(X, \sigma\left(X\right)\right) < 0$, it is clear that the map $\gamma^{S^{-1}}: \fa \mapsto \mathbb{C}\llbracket\nu\rrbracket$ vanishes identically if $D = D^S$. As $S = \Id$ if and only if $\gamma^{S^{-1}} = 0$, the proof is complete.
\end{proof}

\noindent In this section, we will be interested in a particular class of elements in $\mathbb{H}\left(\ast_\nu\right)$ which appeared previously.

\begin{df}
For $D \in \mathbb{H}\left(\ast_\nu\right)$, an operator $T \in \op$ such that $T\left(\ast_\nu\right)$ is a $G$-invariant star-product on $\mathbb{D}$ is called a \emph{$D$-retract} if $D = D^T$. We say that $D \in \mathbb{H}\left(\ast_\nu\right)$ is \emph{retractable} when $D$ admits a $D$-retract. 
\end{df}

\begin{prop}\label{34}
The space of retractable homomorphisms of \,$\mathbb{H}\left(\ast_\nu\right)$ is parametrized by the space of formal power series with coefficients in \,$\mathbb{C}^r \oplus Z^2\left(\mathbb{D}\right)^G$.
\end{prop}

\begin{proof}
Let's fix $T \in \op$ such that $T\left(\ast_\nu\right)$ is a $G$-invariant star-product on $\mathbb{D} \simeq \mathbb{S}$. Let's consider $\mathbb{H}^T\left(\ast_\nu\right)$ the space of retractable homomorphisms $D \in \mathbb{H}\left(\ast_\nu\right)$ of the form $D = D^{T^\prime}$ for $T^{\prime} \in \op$ such that $T^{\prime}\left(\ast_\nu\right)$ is $G$-invariant and $G$-equivalent to $T\left(\ast_\nu\right)$. From lemma \ref{31}, we know that $\mathbb{H}^T\left(\ast_\nu\right)$ coincides with the space of homomorphisms $D \in \mathbb{H}\left(\ast_\nu\right)$ defined by $$D : X \in \fg\, \longmapsto\, D_X = S^{-1} \circ T^{-1} \circ U^{-1} \circ X^\star \circ U \circ T \circ S$$ where $U$ is a $G$-equivalence of $G$-invariant star-products on $\mathbb{S}$ and $S$ a $\mathbb{S}$-automorphism of $\ast_\nu$. Given that the operator $U$ commutes with the action of $G$ on $\mathbb{S}$ in this last expression, we get $$\mathbb{H}^T\left(\ast_\nu\right) = \left\{S^{-1} \circ D^T \circ S \,:\, S \in \Aut\left(\ast_\nu\right)\right\}.$$ Therefore, we deduce from lemmas \ref{17} and \ref{20} that $\mathbb{H}^T\left(\ast_\nu\right)$ is parametrized by $\mathbb{C}^r\llbracket\nu\rrbracket$. The proof is then complete in view of lemma \ref{Gequiv}.
\end{proof}

\begin{rem}\label{rg}
In view of this proof, it becomes clear that this statement describes a parametrization of the space of retractable homomorphisms of $\mathbb{H}\left(\ast_\nu\right)$ by a choice of $\mathbb{S}$-automorphism of $\ast_\nu$ and a choice of $G$-equivalence class of $G$-invariant star-products on $\mathbb{D}$. In particular, we can easily deduce from proposition \ref{32} and expression (\ref{actautS}) that this parametrization of $\mathbb{H}^T\left(\ast_\nu\right)$ by a choice of $\mathbb{S}$-automorphism of $\ast_\nu$ is associated to a parametrization by a choice of a linear map $\Lambda : \fs \rightarrow \mathcal{C}^\infty\left(\mathbb{D}\right)\llbracket\nu\rrbracket$ given in proposition \ref{pqmm}.
\end{rem}

\noindent The following corollaries can be deduced easily from proposition \ref{irr}, remark \ref{33} and proposition \ref{34}.

\begin{cor}\label{corirrr}
As a Hermitian symmetric space of non compact type, if the symmetric bounded domain $\mathbb{D}$ is irreducible, then the space of retractable homomorphisms of \,$\mathbb{H}\left(\ast_\nu\right)$ is parametrized by $\mathbb{C}^{r+1}\llbracket\nu\rrbracket$.
\end{cor}

\begin{cor}\label{corirrr2}
The dimension of \,$\mathbb{H}\left(\ast_\nu\right)$ as vector space over \,$\mathbb{C}\llbracket\nu\rrbracket$ is greater than $r + \dim\left(Z^2\left(\mathbb{D}\right)^G\right)$ and this lower bound is reached if and only if each homomorphism of \,$\mathbb{H}\left(\ast_\nu\right)$ is retractable.
\end{cor}

\subsection{PDE's for the retract}

\noindent We now describe our quantization method based on the above-mentioned \emph{retract method}, as well as its interaction with tools that we developed for simplifying underlied computations.

\vspace{-4 mm}

\begin{enumerate}[(i)]
\item \label{eni} The first step is to fix explicitly a linear map $\Lambda : \fs \rightarrow \mathcal{C}^\infty\left(\mathbb{D}\right)\llbracket\nu\rrbracket$ given by proposition \ref{pqmm}. In virtue of proposition \ref{scov}, if the star-product $\ast_\nu$ is $\fs$-covariant, we can simply choose $\Lambda = \lambda$.
\item \label{enii} Then, we compute the set of quantum moment maps that coincide on $\fs$ with the map $\Lambda$ chosen in (\ref{eni}) and that are associated with retractable homomorphisms of $\mathbb{H}\left(\ast_\nu\right)$. The best way for doing that is to solve the equations of remark \ref{33}. The corollaries \ref{corcov} and \ref{corsm} can be helpful if the star-product $\ast_\nu$ is either $\fs$-covariant or $\mathbb{S} M$-invariant.
\begin{rem}\label{2020}
From remark \ref{rg}, we note that the space of such quantum moment maps is parametrized by $Z^2\left(\mathbb{D}\right)^G\llbracket\nu\rrbracket$. In particular, a choice of parameter $\omega_\nu \in Z^2\left(\mathbb{D}\right)^G\llbracket\nu\rrbracket$ will correspond to a choice of $G$-equivalence class of $G$-invariant star-products comprising the Fedosov star-product $\ast_\nu^{\left(\nabla, \omega_\nu\right)}$.
\end{rem}
\item \label{eniii} For each retractable homomorphism $D \in \mathbb{H}\left(\ast_\nu\right)$ associated to a quantum moment map obtained in (\ref{enii}), we compute the set of $D$-retract, ie. the set of invertible linear convolution operator $T \in \Op^{\mathbb{S}}\left(\ast_\nu\right)$ of the form (\ref{conv}) such that $T \circ D_X \circ T^{-1} = X^\star$ for all $X \in \fg$.
\begin{rem}\label{2017}
If the operator $T$ satisfies this last equation for two elements in $\fg$, then it satisfies also this equation for any multiple of these two elements and for the Lie bracket of these two elements.
\end{rem}
\end{enumerate}

\vspace{-2 mm}

Proposition \ref{fed0}, lemma \ref{31}, lemma \ref{20}, proof of proposition \ref{34} and, more generally, the whole material developed through this section, lead us to the following major result.

\begin{thm}
This method yields the set of $G$-invariant star-products on $\mathbb{D}$. In particular, the set of $D$-retract for a retractable homomorphism $D \in \mathbb{H}\left(\ast_\nu\right)$ associated to a quantum moment map from (\ref{enii}) generates exactly a $G$-equivalence class of $G$-invariant star-products on $\mathbb{D}$.
\end{thm}


In order to reach the objectives stated in our introduction, we need to develop tools for the resolution of the equations expressed in point (\ref{eniii}) of the method. An important step in this direction lies in the following theorem which was firstly proved for $r = 1$ in collaboration with Bieliavsky ; \cite[Ch.\,2, thm.\,2.5.10]{Ko14}.

\begin{thm}\label{thm} 
Let's consider a retractable homomorphism $D \in \mathbb{H}\left(\ast_\nu\right)$. An operator $T \in \Op^{\mathbb{S}}\left(\ast_\nu\right)$ is a $D$-retract if and only if it is the inverse of a convolution operator of the form (\ref{conv}) such that its kernel \,$v_T := u_{T^{-1}} \in \mathcal{D}^\prime\left(\mathbb{S}\right)\llbracket\nu\rrbracket$ satisfies the partial differential equation 
\begin{eqnarray} \label{eq}
D_X\left(v_T\right) = \left(\left[X\right]_\mathfrak{s}\right)^\star\left(v_T\right)
\end{eqnarray} 
for each $X \in \fg$.
\end{thm}

\begin{rem}\label{999}
Equation (\ref{eq}) is trivially satisfied for each $X \in \fs$. In particular, we have
\begin{eqnarray} \nonumber
D_X\left(v_T\right) = \left(\left[X\right]_\mathfrak{s}\right)^\star\left(v_T\right) \text{ \,for all $X \in \fg$} \,\,\,\Longleftrightarrow\,\,\, D_X\left(v_T\right) = 0 \text{ \,for all $X \in \fk$.}
\end{eqnarray}
\end{rem}

\begin{proof}
Let $\Lambda : \fg \rightarrow \mathcal{C}^\infty\left(\mathbb{D}\right)\llbracket\nu\rrbracket : X \mapsto \Lambda_X = \lambda_X + o\left(\nu\right)$ be a quantum moment map associated to $D$ by lemma \ref{5}. 
Let $T \in \Op^{\mathbb{S}}\left(\ast_\nu\right)$ be a $\mathbb{S}$-equivalence of star-products on $\mathbb{D} \simeq \mathbb{S}$. In particular, its inverse $T^{-1}$ is an invertible linear convolution operators on $\mathcal{C}^\infty\left(\mathbb{S}\right)\llbracket\nu\rrbracket$ of the form (\ref{conv}) with kernel $$v_T := u_{T^{-1}} \in \mathcal{D}^\prime\left(\mathbb{S}\right)\llbracket\nu\rrbracket.$$ Let's denote by $ds$ the left-invariant Haar measure on $\mathbb{S}$. For $X \in \fg$, functional analysis theory on the left-invariant Lie group $\left(\mathbb{S}, ds\right)$ allows us to obtain formally the equivalences :
\begin{eqnarray} \nonumber D_X \circ T^{-1} = T^{-1} \circ X^\star &\Leftrightarrow& \left(D_X\right)_{s_0} \left(\int_{\mathbb{S}} v_T\left(s^{-1} s_0\right) f\left(s\right) ds\right) \,=\, \int_{\mathbb{S}} v_T\left(s^{-1} s_0\right) X^\star_s \left(f\right) ds \\Ê\nonumber && \text{for all $s_0 \in \mathbb{S}$ and $f \in \mathcal{D}\left(\mathbb{S}\right)$} \\Ê\nonumber &\Leftrightarrow& \int_{\mathbb{S}} \left(D_X\right)_{s_0} \left(v_T\left(s^{-1} s_0\right)\right) f\left(s\right) ds \,=\, - \int_{\mathbb{S}} X^\star_s\left(v_T\left(s^{-1} s_0\right)\right) f\left(s\right) ds \\Ê\nonumber && \text{for all $s_0 \in \mathbb{S}$ and $f \in \mathcal{D}\left(\mathbb{S}\right)$} \\Ê\label{retract} &\Leftrightarrow& \left(D_X\right)_{s_0} \left(v_T\left(s^{-1} s_0\right)\right) \,=\, - X^\star_s\left(v_T\left(s^{-1} s_0\right)\right) \text{ \, for all $s_0, s \in \mathbb{S}$.} \,\,\,\,\,\,\,\,\,\,\,\,
\end{eqnarray}

\vspace{-2 mm}

\noindent For each ${s_0}, s \in \mathbb{S}$ and $X \in \fg$, given that the star-product $\ast_\nu$ is $\mathbb{S}$-invariant, we have
\begin{eqnarray}
\nonumber \left(D_X\right)_{s_0}\left(v_T\left(s^{-1} {s_0}\right)\right) &=& \frac{1}{2 \nu} \,\left[\Lambda_X, L_{s^{-1}}^\star\left(v_T\right)\right]_{\ast_\nu}\left({s_0}\right) \\ \nonumber &=& \frac{1}{2 \nu} \, \left(L_{s^{-1}}^\star\left(\left[L_{s}^\star \left(\Lambda_X\right), v_T\right]_{\ast_\nu}\right)\right)\left({s_0}\right) \\ \label{retract2} &=& \frac{1}{2 \nu} \, \left[\Lambda_{\Ad_{s^{-1}}\left(X\right)}, v_T\right]_{\ast_\nu} \left(s^{-1} {s_0}\right) \,\,\,=\,\,\, \left(D_{\Ad_{s^{-1}}\left(X\right)}\right)_{s^{-1} {s_0}}\left(v_T\right)
\end{eqnarray}
where the third equality comes from the relation $L_{s^{-1}}^\star \circ \Lambda = \Lambda \circ \Ad_{s}$ which is the integral version of equality $Y^\star \circ \Lambda = \Lambda \circ \ad_{Y}$ for $Y \in \fs$; \cite[Ch.\,2, lem.\,2.5.8]{Ko14}.

\noindent For ${s_0} \in \mathbb{S}$, let $i_{\mathbb{S}}$ and $R_{s_0}$ be the maps defined on $\mathbb{S}$ by $i_{\mathbb{S}}\left(s\right) := s^{-1}$ and $R_{s_0}\left(s\right) := s s_0$ respectively. For each ${s_0}, s \in \mathbb{S}$ and $X \in \fg$, by using relations (\ref{2}) we get the equalities
\begin{eqnarray}
\nonumber - \,X_s^\star\left(v_T\left(s^{-1} {s_0}\right)\right) &=& -\left.\frac{d}{dt}\right|_0 \,v_T\left(\left(L_{\exp\left(- \,t X\right)}\left(s\right)\right)^{-1} {s_0}\right) \\ \nonumber &=& \left.\frac{d}{dt}\right|_0 \left(\left(R_{s_0} \circ i_{\mathbb{S}} \circ L_s\right)^\star v_T\right)\left(\left[\exp\left(t \Ad_{s^{-1}}\left(X\right)\right)\right]_\mathbb{S}\right) \\ \nonumber &=& \left.\frac{d}{dt}\right|_0 \left(\left(R_{s_0} \circ i_{\mathbb{S}} \circ L_s\right)^\star v_T\right)\left(\exp\left(t \left[\Ad_{s^{-1}}\left(X\right)\right]_\mathfrak{s}\right)\right) \\ \nonumber &=& \left.\frac{d}{dt}\right|_0 \,v_T\left(\exp\left(-\,t \left[\Ad_{s^{-1}}\left(X\right)\right]_\mathfrak{s}\right) s^{-1} {s_0}\right) \\ \label{retract1} &=& \left.\frac{d}{dt}\right|_0 \,v_T\left(L_{\exp\left(- \,t\, \left[\Ad_{s^{-1}}\left(X\right)\right]_\mathfrak{s}\right)}\left(s^{-1} {s_0}\right)\right) \,\,\,=\,\,\, \left(\left[\Ad_{s^{-1}}\left(X\right)\right]_\mathfrak{s}\right)_{s^{-1} {s_0}}^\star \left(v_T\right). \,\,\,\,\,\,\,\,\,\,\,\,
\end{eqnarray}
If we combine relations (\ref{retract}), (\ref{retract2}) and (\ref{retract1}), we obtain the equivalence
\begin{eqnarray} \label{2018}
D_X \circ T^{-1} = T^{-1} \circ X^\star \,\,\Longleftrightarrow\,\, \left(D_{\Ad_{s^{-1}}\left(X\right)}\right)\left(v_T\right) = \left(\left[\Ad_{s^{-1}}\left(X\right)\right]_\mathfrak{s}\right)^\star \left(v_T\right) \text{ \,for all $s \in \mathbb{S}$.}
\end{eqnarray}
As the operator $T \in \Op^{\mathbb{S}}\left(\ast_\nu\right)$ is a $D$-retract if and only if $D_X \circ T^{-1} = T^{-1} \circ X^\star$ for all $X \in \fg$, the proof is complete.
\end{proof}

\noindent As we will see in the next section, it can be hard to solve such a hierarchy of partial differential equations on the kernel of an invertible linear convolution operator. The fact that the homomorphism $D \in \mathbb{H}\left(\ast_\nu\right)$ is retractable is crucial to ensure that there exist solutions to these equations. 

\noindent The following results are also helpful tools for the resolution of these equations on specific examples. The first proposition can be deduced from lemma \ref{m}, remark \ref{mm} and corollary \ref{corsm}; \cite[cor.\,3]{Ko16}. 

\begin{prop}\label{smx} Let's assume that the star-product $\ast_\nu$ is \,$\mathbb{S} M$-invariant and \,$\mathbb{S} M$-equivalent to a $G$-invariant star-product on $\mathbb{D}$. Then, equation (\ref{eq}) is satisfied for each $X \in \fs \oplus \fm$ if and only if\, $Y^\star\left(v_T\right) = 0$ for all \,$Y \in \fm$. In this case, the map $$X \in \fg_\lambda \mapsto v_T\left(\exp\left(X\right)\right)$$ is radial in the Euclidian vector spaces $\left(\fg_\lambda, \beta_\sigma\right)$ and $\left(g_\lambda, \left(- | -\right)\right)$ for any $\lambda \in \Sigma^+$.
\end{prop}

\noindent This first proposition can advantageously be used to reduce the number of variables in the equations. 

\noindent The second proposition is a consequence of remark \ref{2017} and relation (\ref{2018}).

\begin{prop}\label{propw}
Let $W$ be a vector space such that $$\fs \subseteq W \subseteq \fg \text{\, \, and \, \,} \left\{\Ad_{s^{-1}}\left(X\right) : X \in W, \,sÊ\in \mathbb{S}\right\} = W.$$ If equation (\ref{eq}) is satisfied for all $X \in W$, then it is also satisfied for all $X \in \left[W, W\right]$.
\end{prop}

\noindent The major advantage of this second proposition is to reduce the number of equations to be solved. We choose this formulation for this proposition because of its correspondence with an efficient computation method based on the root space decomposition of $\fg$. Otherwise, it can be expressed obviously from remark \ref{999}, as $\left.D\right|_{\fk}$ is a Lie algebra homomorphism on $\fk$.

\section{Example: the complex unit ball}

For $n \geq 1$, the goal of this section is to develop the above-mentioned method on the complex unit ball in $\mathbb{C}^N$ as symmetric bounded domain. We will denote it by $\mathbb{D}_N$. This section is entirely based on the reference \cite{Ko14} in which the reader will find the computational details of the results described hereafter. We will refer continuously to the notations introduced in this text.

\noindent Let's describe the geometry of the domain $\mathbb{D}_N$ and the Lie algebra structure of its automorphism group from propositions \ref{rsd}, \ref{iwa2}, \ref{cool} and references \cite[Ch.\,1, \S\,1.5]{Ko14}, \cite[Ch.\,10, \S\,6.3]{He01}. 

\begin{prop} \hspace{1 mm}

\vspace{-5 mm}

\begin{enumerate}[(a)]
\item The complex unit ball \,$\mathbb{D}_N$ admits a structure of rank one irreducible Hermitian symmetric space of non compact type with automorphism group $G = SU\left(1, N\right)$. 
\item The Iwasawa group \,$\mathbb{S}$ of $G$ is an elementary Pyatetskii-Shapiro group acting simply transitively on $\mathbb{D}_N$. 
\item The structure of its Lie algebra $$\mathbb{R}^{2N} \,\simeq\, \fs \,=\, \mathfrak{a} \ltimes \mathfrak{n} \,=\, \mathbb{R} H \ltimes \left(V \oplus \mathbb{R} E\right)$$ is described in lemma \ref{psds} and the root space decomposition of the Lie algebra $\fg$ is of the form 
$$\fg \,=\, \fg_{2\lambda} \,\oplus\, \fg_{\lambda} \,\oplus\, \fg_0 \,\oplus\, \fg_{-\lambda} \,\oplus\, \fg_{-2\lambda} \text{\, for \,} \lambda \in \Sigma^+$$
with \,$\fg_{2\lambda} = \mathbb{R} E, \,\,\,\fg_{\lambda} = V, \,\,\,\fg_{0} \,=\, \mathbb{R} H \oplus \fm \,\simeq\, \mathbb{R} H \oplus \mathfrak{u}\left(N-1\right), \,\,\,\fg_{-\lambda} = \sigma\left(V\right)\,\, \text{and} \,\,\fg_{-2\lambda} = \mathbb{R} \,\sigma\left(E\right).$

\item The map
$$\mathbb{R}^{2N} \simeq\, \fs \,\rightarrow\, \mathbb{S} \,:\, \left(a, v^1_1, ..., v^1_{N-1}, v^i_1, ..., v^i_{N-1}, z\right) \simeq\, a H + v + z E \,\mapsto\, \exp\left(a H\right) \,\exp\left(v\right) \,\exp\left(z E\right)$$
defines a global Darboux chart on $\mathbb{S} \simeq \mathbb{D}_N$, where the vector $\left(v^1_1, ..., v^i_{N-1}\right)$ is the decomposition of $v$ in a symplectic basis of $\left(V, \Omega\right)$.
\end{enumerate}
\end{prop}

Given that $\rank\left(\mathbb{D}_N\right) = 1$, it is clear from corollary \ref{H2S} that all the $\mathbb{S}$-invariant star-products on $\mathbb{D}_N$ are $\mathbb{S}$-equivalent. In particular, in this case, the statement of proposition \ref{fed0} is trivial. From proposition \ref{irr}, we deduce that the $G$-equivalence classes of $G$-invariant star products on $\mathbb{D}_N$ are parametrized by $Z^2\left(\mathbb{D}\right)^G\llbracket\nu\rrbracket \simeq \mathbb{C}\llbracket\nu\rrbracket$. As a consequence, all the $G$-invariant star-products on $\mathbb{D}_N$ are $G$-equivalent up to a reparametrization of the formal parameter $\nu$; \cite{BB03}.

\noindent Let's consider the $\mathbb{S}$-invariant star-product $\ast_\nu$ on $\mathbb{S} \simeq \mathbb{D}_N$ obtained from the explicit deformation quantization described in \cite[Ch.\,4, thm.\,4.5]{BG15}, \cite[Ch.\,3, thm.\,3.1]{BM01} and \cite[Ch.\,2, thm.\,2.4.7]{Ko14}. 

\noindent From computations made in reference \cite[Ch.\,2, \S\,2.3.4 and \S\,2.4.1]{Ko14}, it is proven that the Moyal-Weyl star-product on $\fs \simeq \mathbb{R}^{2N}$ is $\fs$-covariant and equivalent to $\ast_\nu$ by an operator $T_0$ preserving the moment map on $\fs$; \cite{BG15}, \cite{BM01}. As a consequence, the star-product $\ast_\nu$ is $\fs$-covariant. 
In addition, elementary computations from reference \cite[Ch.\,2, \S\,2.6.3]{Ko14} give us the relation $$\frac{1}{2 \nu} \left[\lambda_Y, -\right]_{\ast_\nu} = Y^\star$$ for each $Y \in \fm$. As a consequence, corollaries \ref{corcov}, \ref{corsm} and \ref{corirrr} lead us to the following result.

\begin{lem}
The star-product $\ast_\nu$ is $\fs$-covariant, $\mathbb{S} M$-invariant and $\mathbb{S} M$-equivalent to any $G$-invariant star-product on $\mathbb{D}_N$. If $D \in \mathbb{H}\left(\ast_\nu\right)$ is a retractable homomorphism and $\Lambda$ the quantum moment map associated to $D$ by lemma \ref{5}, then $$D_X = X^\star \text{ \, and \, } \Lambda_X - \lambda_X \in \nu\, \mathbb{C}\llbracket\nu\rrbracket \text{ \, for each $X \in \fs \oplus \fm$.}$$ In addition, the space of retractable homomorphisms of \,$\mathbb{H}\left(\ast_\nu\right)$ is parametrized by $\mathbb{C}^{2}\llbracket\nu\rrbracket$.  
\end{lem}
 
As the star-product $\ast_\nu$ is $\fs$-covariant, we can fix the moment map $\lambda$ as the linear map on $\fs$ given by proposition \ref{pqmm}. 
After some computations, we obtain the following expressions in the global Darboux coordinate system $\left(a, v, z\right)$.

\begin{lem}\label{qmmex}
Let $\Lambda : \fg \rightarrow \mathcal{C}^\infty\left(\mathbb{S}\right)\llbracket\nu\rrbracket$ be the quantum moment map associated to a homomorphism of \,$\mathbb{H}\left(\ast_\nu\right)$ by lemma \ref{5}. If $\left.\Lambda\right|_{\fs} = \lambda$, then there exists a formal constant $\alpha \in \mathbb{C}\llbracket\nu\rrbracket$ such that 
\begin{eqnarray} \nonumber
&& T_0 \Lambda_{Y} = \lambda_{Y} \text{ \, for all \,$Y \in \left[\fm, \fm\right]$ \,} \text{ \, and \, } T_0 \Lambda_{Z} = \lambda_{Z} + \alpha \text{ \, for all \,$Z \in Z\left(\fm\right)$;} \\ \nonumber
&& \left(T_0 \Lambda_{\sigma\left(v_0\right)}\right)\left(a, v, z\right) = e^a \left[4 \left(v_0\text{\hspace{0.1 mm}}|\text{\hspace{0.1 mm}}v\right) z - \left(\left(v\text{\hspace{0.1 mm}}|\text{\hspace{0.1 mm}}v\right) + \alpha\right) \,\Omega\left(v_0, v\right)\right] \text{ \, for all \,$v_0 \in V$;} \\ \nonumber
&& \left(T_0 \Lambda_{\sigma\left(E\right)}\right)\left(a, v, z\right) = e^{2a} \left[4\text{\hspace{0.2 mm}}z^2 + \left(\left(v\text{\hspace{0.1 mm}}|\text{\hspace{0.1 mm}}v\right) + \alpha\right)^2 + \left(N-1\right) \nu^2\right];
\end{eqnarray}
where $\left(- | -\right)$ is the scalar product on $\fs$ induced by the K\"ahlerian structure of \,$\mathbb{S} \simeq \mathbb{D}_N$ and described in remark \ref{mm}.
\end{lem}

\noindent In view of this result, corollary \ref{corirrr2} and remark \ref{2020} lead us to the following statement.

\begin{cor}
Every homomorphism of \,$\mathbb{H}\left(\ast_\nu\right)$ is retractable.
\end{cor}

\noindent Let's consider a homomorphism $D \in \mathbb{H}\left(\ast_\nu\right)$ associated to a quantum moment map $\Lambda$ obtained in lemma \ref{qmmex}. In particular, the previous corollary ensures the existence of solutions to the hierarchy of partial differential equations (\ref{eq}) characterizing the kernel $v_T \in \mathcal{D}^\prime\left(\mathbb{S}\right)\llbracket\nu\rrbracket$ of the inverse of any $D$-retract. 

\noindent Let's consider the vector space $$W = \fs \oplus \fm \oplus \fg_{-\lambda}.$$ As it satisfies the hypothesis of proposition \ref{propw}, the kernel $v_T$ is solution to PDE (\ref{eq}) for all $X \in W$ if and only if it is solution to PDE (\ref{eq}) for all $X \in \fg = \left[W, W\right]$; \cite[Ch.\,2, lem.\,2.6.5]{Ko14}. Therefore, it is not necessary to write and to solve the PDE (\ref{eq}) for $X \in \fg_{-2\lambda}$ if we solve it for all $X \in W$. This result is very important because it reduces the number of equations that we have to consider. This case is discussed in \cite[Appendix B \& Ch.\,2, \S\,2.7]{Ko14}. 

\noindent Moreover, an application of proposition \ref{smx} with our choice Darboux coordinate system proves that $v_T$ depends only on $a$, $z$ and the radial component $r$ of $v$ in the Euclidian vector space $$\left(g_\lambda = V,\, \left(- | -\right)\right).$$ It is easy to check that any such solution of the form $v_T\left(a, r, z\right)$ satisfies equation (\ref{eq}) for $X \in \fs \oplus \fm$; \cite[Ch.\,2, lem.\,2.6.9]{Ko14}. As a consequence, we obtain the following lemma.

\begin{lem}
An operator $T \in \Op^{\mathbb{S}}\left(\ast_\nu\right)$ is a $D$-retract if and only if it is the inverse of a convolution operator with a kernel of the form $v_T\left(a, r, z\right)$ satisfying the partial differential equation 
\begin{eqnarray}\label{eqfin}
D_{\sigma\left(v_0\right)}\left(v_T\right) + \left(v_0\right)^\star v_T = 0
\end{eqnarray}
for each $v_0 \in V$. 
\end{lem}

\noindent In order to simplify the resolution of this equation, we can fix $\alpha = 1$ in quantum moment map $\Lambda$. In fact, the computation of the set of $D$-retract for only one such homomorphism $D$ is enough to determine the set of $G$-invariant star-products on $\mathbb{D}$ up a reparametrization of the formal parameter $\nu$. 

\noindent For computational reasons and without loss of generality, equation (\ref{eqfin}) was intertwined by a partial Fourier transform $\mathcal{F}$ in the $z$ variable and written with the notation $\sigma\left(v_0\right) = \left[\sigma\left(E\right), w\right]$ for $w \in V$; \cite[Ch.\,2, \S\,2.6]{Ko14}. We obtained 
\begin{eqnarray}
\nonumber 0 & = & \left[i \, \xi \, e^{a} \left[\left(1+ \sqrt{1 - \nu^2 \xi^2}\right) r^2 \, + \, 2\,+\, 2\,i\,\xi \, e^{-a}\right]  {\left(w | v\right)}\right] \vartheta \\ \nonumber && - \left[e^{a} \left(1+ \sqrt{1 - \nu^2 \xi^2}\right)  {\Omega\left(w, v\right)}\right] \vartheta \\ \nonumber & & - \left[e^{a} \left(1+ \sqrt{1 - \nu^2 \xi^2}\right)  {\Omega\left(w, v\right)}\right] \, \partial_a\left(\vartheta\right) \\ \nonumber & & + \left[e^{a} \left(- 1+ \sqrt{1 - \nu^2 \xi^2}\right) \,-\, \frac{e^{-a}}{r}\right] \, {\Omega\left(w, v\right)}\, r \, \partial_r \left(\vartheta\right) \\ \nonumber & & + \left[e^{a} \left[\left(1+ \sqrt{1 - \nu^2 \xi^2}\right) r^2 \, + \, 2\right]\right] \, \frac{ {\Omega\left(w, v\right)}}{2 r} \,\, \partial_r \left(\vartheta\right) \\ \nonumber & & - \left[2\, e^{a} \,\xi\, \sqrt{1 - \nu^2 \xi^2} \,\,  {\Omega\left(w, v\right)}\right] \, \partial_\xi\left(\vartheta\right) \\ \nonumber & & - \left[\frac{i \, e^{a}}{\xi} \left(-1+ \sqrt{1 - \nu^2 \xi^2}\right) {\left(w | v\right)}\right] \left(\partial_r^2 \left(\vartheta\right) \,+\, \frac{2 n - 3}{r} \,\, \partial_r \left(\vartheta\right)\right) \\ \nonumber & & + \left[\frac{2\,i \, e^{a}}{\xi} \left(-1+ \sqrt{1 - \nu^2 \xi^2}\right)\right] \,  {\left(w | v\right)} \, \partial_r^2 \left(\vartheta\right) \\ \nonumber & & - \left[\frac{2\,i \, e^{a}}{\xi} \left(-1+ \sqrt{1 - \nu^2 \xi^2}\right)\right] \, \frac{ {\left(w | v\right)}}{r} \,\, \partial_r \left(\vartheta\right) \\ \nonumber & & - \left[\frac{2\,i \, e^{a}}{\xi} \left(-1+ \sqrt{1 - \nu^2 \xi^2}\right)\right] \, \frac{ {\left(w | v\right)}}{r} \,\, \partial_a\left(\partial_r \left(\vartheta\right)\right) \\ \nonumber & & - \left[4\,i\,e^a\,\sqrt{1 - \nu^2 \xi^2}\right] \, \frac{ {\left(w | v\right)}}{r} \,\, \partial_\xi\left(\partial_r \left(\vartheta\right)\right) \\ \nonumber & & - \left[\frac{e^a}{\xi^2} \, \left(-1+ \sqrt{1 - \nu^2 \xi^2}\right)\right] \, \frac{ {\Omega\left(w, v\right)}}{2 r} \left[\frac{2 n - 3}{r} \,\, \partial_r^2 \left(\vartheta\right) \,-\, \frac{2 n - 3}{r^2} \,\, \partial_r \left(\vartheta\right)\right] \\ \nonumber & & - \left[\frac{e^a}{\xi^2} \, \left(-1+ \sqrt{1 - \nu^2 \xi^2}\right)\right] \, \frac{ {\Omega\left(w, v\right)}}{2 r} \,\,\partial^3_r \left(\vartheta\right).
\end{eqnarray}
where $\vartheta = \mathcal{F}\left(v_T\right)\left(a, r, \xi\right)$. The resolution of this equation is widely discussed in reference \cite[Ch.\,2, \S\,2.6.5]{Ko14} where various changes of variables and integral transforms are used. 


\end{document}